\documentclass{amsart}
\usepackage{latexsym}
\usepackage{amssymb, amsbsy, amsthm, amsmath, amstext, amsopn, verbatim}
\usepackage[color,all]{xy}
\usepackage{amsfonts}
\usepackage{amscd}
\usepackage{mathrsfs}
\usepackage{verbatim}
\usepackage{xcolor}
\usepackage{tabularx}
\usepackage{soul}

\input xy
\xyoption{all}

\usepackage{parcolumns}

\newtheorem{thm}{Theorem} [section]
\newtheorem{lemma}[thm]{Lemma}

\newtheorem{corollary}[thm]{Corollary}
\newtheorem{prop}[thm]{Proposition}
\newtheorem{notation}[thm]{Notation}

\newtheorem*{rough-thm-1}{Rough Version of Vanishing Theorem}
\newtheorem*{rough-thm-2}{Rough Version of Exactness Theorem}

\newtheorem*{main example}{Main Example}

\theoremstyle{definition}

\newtheorem*{basic convention}{Basic Conventions}

\newtheorem{defn}[thm]{Definition}

\theoremstyle{remark}

\newtheorem{remark}[thm]{Remark}

\begin{document}

\numberwithin{equation}{section}

\newcommand{\hs}{\mbox{\hspace{.4em}}}
\newcommand{\ds}{\displaystyle}
\newcommand{\bd}{\begin{displaymath}}
\newcommand{\ed}{\end{displaymath}}
\newcommand{\bcd}{\begin{CD}}
\newcommand{\ecd}{\end{CD}}

\newcommand{\proj}{\operatorname{Proj}}
\newcommand{\bproj}{\underline{\operatorname{Proj}}}
\newcommand{\spec}{\operatorname{Spec}}
\newcommand{\bspec}{\underline{\operatorname{Spec}}}
\newcommand{\pline}{{\mathbf P} ^1}
\newcommand{\pplane}{{\mathbf P}^2}
\newcommand{\coker}{{\operatorname{coker}}}
\newcommand{\ldb}{[[}
\newcommand{\rdb}{]]}

\newcommand{\Sym}{\operatorname{Sym}^{\bullet}}
\newcommand{\Symp}{\operatorname{Sym}}
\newcommand{\Pic}{\operatorname{Pic}}
\newcommand{\AAut}{\operatorname{Aut}}
\newcommand{\PAut}{\operatorname{PAut}}

\newcommand{\too}{\twoheadrightarrow}
\newcommand{\C}{{\mathbf C}}
\newcommand{\cA}{{\mathcal A}}
\newcommand{\cS}{{\mathcal S}}
\newcommand{\cV}{{\mathcal V}}
\newcommand{\cM}{{\mathcal M}}
\newcommand{\bA}{{\mathbf A}}
\newcommand{\aline}{\mathbb{A}^1}
\newcommand{\cB}{{\mathcal B}}
\newcommand{\cC}{{\mathcal C}}
\newcommand{\cD}{{\mathcal D}}
\newcommand{\D}{{\mathcal D}}
\newcommand{\cs}{{\mathbf C} ^*}
\newcommand{\boldc}{{\mathbf C}}
\newcommand{\cE}{{\mathcal E}}
\newcommand{\cF}{{\mathcal F}}
\newcommand{\cG}{{\mathcal G}}
\newcommand{\G}{{\mathbf G}}
\newcommand{\fg}{{\mathfrak g}}
\newcommand{\ft}{\mathfrak t}
\newcommand{\bH}{{\mathbf H}}
\newcommand{\cH}{{\mathcal H}}
\newcommand{\cI}{{\mathcal I}}
\newcommand{\cJ}{{\mathcal J}}
\newcommand{\cK}{{\mathcal K}}
\newcommand{\cL}{{\mathcal L}}
\newcommand{\baL}{{\overline{\mathcal L}}}
\newcommand{\M}{{\mathcal M}}
\newcommand{\bM}{{\mathbf M}}
\newcommand{\bm}{{\mathbf m}}
\newcommand{\cN}{{\mathcal N}}
\newcommand{\theo}{\mathcal{O}}
\newcommand{\cP}{{\mathcal P}}
\newcommand{\cR}{{\mathcal R}}
\newcommand{\boldp}{{\mathbf P}}
\newcommand{\boldq}{{\mathbf Q}}
\newcommand{\bbL}{{\mathbf L}}
\newcommand{\cQ}{{\mathcal Q}}
\newcommand{\cO}{{\mathcal O}}
\newcommand{\Oo}{{\mathcal O}}
\newcommand{\OX}{{\Oo_X}}
\newcommand{\OY}{{\Oo_Y}}
\newcommand{\otY}{{\underset{\OY}{\ot}}}
\newcommand{\otX}{{\underset{\OX}{\ot}}}
\newcommand{\cU}{{\mathcal U}}
\newcommand{\cX}{{\mathcal X}}
\newcommand{\cW}{{\mathcal W}}
\newcommand{\boldz}{{\mathbf Z}}
\newcommand{\cZ}{{\mathcal Z}}
\newcommand{\qgr}{\operatorname{qgr}}
\newcommand{\gr}{\operatorname{gr}}
\newcommand{\coh}{\operatorname{coh}}
\newcommand{\End}{\operatorname{End}}
\newcommand{\Hom}{\operatorname{Hom}}
\newcommand{\uHom}{\underline{\operatorname{Hom}}}
\newcommand{\uHomY}{\uHom_{\OY}}
\newcommand{\uHomX}{\uHom_{\OX}}
\newcommand{\Ext}{\operatorname{Ext}}
\newcommand{\bExt}{\operatorname{\bf{Ext}}}
\newcommand{\Tor}{\operatorname{Tor}}

\newcommand{\inv}{^{-1}}
\newcommand{\airtilde}{\widetilde{\hspace{.5em}}}
\newcommand{\airhat}{\widehat{\hspace{.5em}}}
\newcommand{\nt}{^{\circ}}
\newcommand{\del}{\partial}

\newcommand{\supp}{\operatorname{supp}}
\newcommand{\GK}{\operatorname{GK-dim}}
\newcommand{\hd}{\operatorname{hd}}
\newcommand{\id}{\operatorname{id}}
\newcommand{\res}{\operatorname{res}}
\newcommand{\lrar}{\leadsto}
\newcommand{\im}{\operatorname{Im}}
\newcommand{\HH}{\operatorname{H}}
\newcommand{\TF}{\operatorname{TF}}
\newcommand{\Bun}{\operatorname{Bun}}
\newcommand{\Hilb}{\operatorname{Hilb}}
\newcommand{\Fact}{\operatorname{Fact}}
\newcommand{\F}{\mathcal{F}}
\newcommand{\nthord}{^{(n)}}
\newcommand{\Aut}{\underline{\operatorname{Aut}}}
\newcommand{\Gr}{\operatorname{Gr}}
\newcommand{\Fr}{\operatorname{Fr}}
\newcommand{\GL}{\operatorname{GL}}
\newcommand{\gl}{\mathfrak{gl}}
\newcommand{\SL}{\operatorname{SL}}
\newcommand{\ff}{\footnote}
\newcommand{\ot}{\otimes}
\def\Ext{\operatorname {Ext}}
\def\Hom{\operatorname {Hom}}
\def\Ind{\operatorname {Ind}}
\def\bbZ{{\mathbb Z}}

\newcommand{\nc}{\newcommand}
\newcommand{\on}{\operatorname}
\nc{\cont}{\on{cont}}
\nc{\rmod}{\on{mod}}
\nc{\Mtil}{\widetilde{M}}
\nc{\wb}{\overline}
\nc{\wt}{\widetilde}
\nc{\wh}{\widehat}
\nc{\sm}{\setminus}
\nc{\mc}{\mathcal}
\nc{\mbb}{\mathbb}
\nc{\Mbar}{\wb{M}}
\nc{\Nbar}{\wb{N}}
\nc{\Mhat}{\wh{M}}
\nc{\pihat}{\wh{\pi}}
\nc{\JYX}{\cJ_{Y\leftarrow X}}
\nc{\phitil}{\wt{\phi}}
\nc{\Qbar}{\wb{Q}}
\nc{\DYX}{\D_{Y\leftarrow X}}
\nc{\DXY}{\D_{X\to Y}}
\nc{\dR}{\stackrel{\bbL}{\underset{\D_X}{\ot}}}
\nc{\Winfi}{\cW_{1+\infty}}
\nc{\K}{{\mc K}}
\nc{\unit}{{\bf \on{unit}}}
\nc{\boxt}{\boxtimes}
\nc{\xarr}{\stackrel{\rightarrow}{x}}
\nc{\Cnatbar}{\overline{C}^{\natural}}
\nc{\oJac}{\overline{\on{Jac}}}
\nc{\gm}{{\mathbf G}_m}
\nc{\Loc}{\on{Loc}}
\nc{\Bm}{\operatorname{Bimod}}
\nc{\lie}{{\mathfrak g}}
\nc{\lb}{{\mathfrak b}}
\nc{\lien}{{\mathfrak n}}
\nc{\e}{\epsilon}
\nc{\eu}{\mathsf{eu}}
\nc{\Q}{\mathbb{Q}}

\nc{\Gm}{{\mathbb G}_m}
\nc{\Gabar}{\wb{\G}_a}
\nc{\Gmbar}{\wb{\G}_m}
\nc{\PD}{{\mathbb P}_{\D}}
\nc{\Pbul}{P_{\bullet}}
\nc{\PDl}{{\mathbb P}_{\D(\lambda)}}
\nc{\PLoc}{\mathsf{MLoc}}
\nc{\Tors}{\on{Tors}}
\nc{\PS}{{\mathsf{PS}}}
\nc{\PB}{{\mathsf{MB}}}
\nc{\Pb}{{\underline{\operatorname{MBun}}}}
\nc{\Ht}{\mathsf{H}}
\nc{\bbH}{\mathbb H}
\nc{\gen}{^\circ}
\nc{\Jac}{\operatorname{Jac}}
\nc{\sP}{\mathsf{P}}
\nc{\sT}{\mathsf{T}}
\nc{\bP}{{\mathbb P}}
\nc{\otc}{^{\otimes c}}
\nc{\Det}{\mathsf{det}}
\nc{\PL}{\on{ML}}
\nc{\sL}{\mathsf{L}}

\nc{\ml}{{\mathcal S}}
\nc{\Xc}{X_{\on{con}}}
\nc{\Z}{{\mathbb Z}}
\nc{\resol}{\mathfrak{X}}
\nc{\map}{\mathsf{f}}
\nc{\gK}{\mathbb{K}}
\nc{\bigvar}{\mathsf{W}}
\nc{\Tmax}{\mathsf{T}^{md}}

\nc{\sA}{\mathsf{A}}
\nc{\sS}{\mathsf{S}}

\nc{\Cpt}{\mathbb{P}}
\nc{\pv}{e}

\nc{\Qdbl}{Q^{\on{dbl}}}
\nc{\Qgtr}{Q^{\on{gtr}}}
\nc{\algtr}{\alpha^{\on{gtr}}}
\nc{\Ggtr}{\mathbb{G}^{\on{gtr}}}
\nc{\chigtr}{\chi^{\on{gtr}}}

\newcommand{\la}{\langle}
\newcommand{\ra}{\rangle}
\newcommand{\fm}{\mathfrak m}
\newcommand{\Ms}{\mathfrak M}
\newcommand{\Ma}{\mathfrak M_0}
\newcommand{\ML}{\mathfrak L}
\newcommand{\ev}{\textit{ev}}
\nc{\thetagtr}{\theta^{\on{gtr}}}
\nc{\bS}{\mathbb{S}}
\nc{\Sgtr}{\mathbb{S}^{\on{gtr}}}

\nc{\intN}{_{[0,N]}}

\numberwithin{equation}{section}

\title{The Pure Cohomology of Multiplicative Quiver Varieties}

\author{Kevin McGerty}
\address[McGerty]{Mathematical Institute\\University of Oxford\\Oxford OX1 3LB, UK}
\email[McGerty]{mcgerty@maths.ox.ac.uk}
\address[McGerty]{Department of Mathematics\\University of Illinois at Urbana-Champaign\\Urbana, IL 61801 USA}
\author{Thomas Nevins}
\address[Nevins]{Department of Mathematics\\University of Illinois at Urbana-Champaign\\Urbana, IL 61801 USA}
\email[Nevins]{nevins@illinois.edu}


\begin{abstract}
To a quiver $Q$ and choices of nonzero scalars $q_i$, non-negative integers $\alpha_i$, and integers $\theta_i$ labeling each vertex $i$, Crawley-Boevey--Shaw associate a {\em multiplicative quiver variety} $\mathcal{M}_\theta^q(\alpha)$, a trigonometric analogue of the Nakajima quiver variety associated to $Q$, $\alpha$, and $\theta$.  We prove that the pure cohomology, in the Hodge-theoretic sense, of the stable locus $\mathcal{M}_\theta^q(\alpha)^{\on{s}}$ is generated as a $\mathbb{Q}$-algebra by the tautological characteristic classes.  In particular, the pure cohomology of genus $g$ twisted character varieties of $GL_n$ is generated by tautological classes.  
\end{abstract}

\maketitle

\section{Introduction}\label{sec:intro}
Let $Q = (I,\Omega)$ be a quiver.  Fix a vector $q\in (\mathbb{C}^\times)^I$.  Associated to these data is a noncommutative algebra $\Lambda^q$, the {\em multiplicative preprojective algebra} \cite{CBS} of $Q$ with parameter $q$.
Letting $\alpha\in \mathbb{Z}_{\geq 0}^I$ be a dimension vector for $Q$ and choosing a stability condition $\theta\in \mathbb{Z}^I$, we get a moduli space $\mathcal{M}_\theta^q(\alpha)$ of $\theta$-semistable representations of $\Lambda^q$ with dimension vector $\alpha$, called a {\em multiplicative quiver variety}, investigated in \cite{CBS, Yamakawa} (and both investigated and substantially generalized in \cite{Boalch}).  Multiplicative quiver varieties provide concrete realizations of character varieties and related spaces: see \cite{BoalchYamakawa, BezKap, ST} among others.

\subsection{Results}
As for its cousins, the Nakajima quiver varieties, the multiplicative quiver 
variety $\mathcal{M}_{\theta}^q(\alpha)$ is defined as a GIT quotient (at a character $\chi_\theta: \mathbb{G}\rightarrow \Gm$) of an affine algebraic variety $\on{Rep}(\Lambda^q,\alpha)$ by the group $\mathbb{G} = \big(\prod_i GL(\alpha_i)\big)/\Delta(\Gm)$, a product of general linear groups modulo the diagonal copy of $\Gm$; when it is a {\em free} quotient, this endows $\mathcal{M}_{\theta}^q(\alpha)$ with a map $\mathcal{M}_{\theta}^q(\alpha)\rightarrow B\mathbb{G}$.  

The rational cohomology $H^*(B\mathbb{G},\mathbb{Q})$ is pure in the sense of Hodge theory: $H^m(B\mathbb{G},\mathbb{Q}) = W_mH^m(B\mathbb{G},\mathbb{Q})$, where $ W_mH^m(B\mathbb{G},\mathbb{Q})$ denotes the weight $m$ part.  Thus, the image of the pullback map on cohomology must land in the pure part, in the Hodge-theoretic sense, namely
 \bd
 PH^*(\mathcal{M}_{\theta}^q(\alpha)) \overset{\operatorname{def}}{=} \ds\bigoplus_m W_mH^m\big(\mathcal{M}_{\theta}^q(\alpha), \mathbb{Q}\big)
 \ed
 of $H^*\big(\mathcal{M}_{\theta}^q(\alpha), \mathbb{Q}\big)$.
The main result of the present paper is:
\begin{thm}\label{main thm}
\mbox{}
\begin{enumerate}
\item
Suppose that $U\subseteq \mathcal{M}_{\theta}^q(\alpha)^{\on{s}}$ is any connected open subset of the stable locus of the multiplicative quiver variety $\mathcal{M}_{\theta}^q(\alpha)$.  Then the induced map on cohomology
\bd
H^*(B\mathbb{G}, \mathbb{Q})\rightarrow H^*\big(U, \mathbb{Q}\big)
\ed
defines a surjection onto the pure cohomology
$PH^*(U) = \ds\bigoplus_m W_mH^m\big(U, \mathbb{Q}\big)$.  
\item
 In particular, if $\mathcal{M}_{\theta}^q(\alpha) = \mathcal{M}_{\theta}^q(\alpha)^{\on{s}}$ and $\mathcal{M}_{\theta}^q(\alpha)$ is connected, then
\bd
H^*(B\mathbb{G}, \mathbb{Q})\rightarrow H^*\big(\mathcal{M}_{\theta}^q(\alpha), \mathbb{Q}\big)
\ed
surjects onto $PH^*\big(\mathcal{M}_{\theta}^q(\alpha)\big)$.
\end{enumerate}
\end{thm}
\noindent
In light of Theorem 1.2 of \cite{McNKirwan}, Theorem \ref{main thm} is nicely consonant with Hausel's ``purity conjecture'' (cf. \cite{Hausel} as well as \cite[Theorem~1.3.1 and Corollary~1.3.2]{HLV}, and the discussion around Conjecture 1.1.3 of \cite{HWW}), which predicts that when 
$\mathcal{M}_{\theta}^q(\alpha) = \mathcal{M}_{\theta}^q(\alpha)^{\on{s}}$, one should have an isomorphism
$PH^*(\mathcal{M}_{\theta}^q(\alpha)^{\on{s}}) \cong H^*\big(\mathfrak{M}_{\theta}(\alpha)^{\on{s}},\mathbb{Q}\big)$, where $\mathfrak{M}_{\theta}(\alpha)^{\on{s}}$ denotes the corresponding Nakajima quiver variety.

In the special case in which $Q$ is a quiver with a single node and $g\geq 1$ loops, the dimension vector is $\alpha=n$, and $q\in \mathbb{C}^\times$ is a primitive $n$th root of unity, the multiplicative quiver variety $\mathcal{M}_{\theta}^q(\alpha)$ is identified with the $GL_n$-character variety $\on{Char}(\Sigma_g, GL_n, q\on{Id})$ of a genus $g$ surface with a single puncture with residue $q\on{Id}$, sometimes called a  genus $g$ {\em twisted character variety} \cite{HR2}.
We obtain:
\begin{corollary}\label{character var}
The pure cohomology $PH^*\big(\on{Char}(\Sigma_g, GL_n, q\on{Id})\big)$ is generated by tautological classes.
\end{corollary}
Corollary \ref{character var} has already appeared in \cite{Shende}, where it was deduced, via the non-abelian Hodge theorem, from Markman's theorem \cite{Markman} that the cohomology of the moduli space of $GL_n$-Higgs bundles of degree $1$ on a smooth projective genus $g$ curve is generated by tautological classes.  A novelty of our result, compared to \cite{Shende}, is that we avoid invoking non-abelian Hodge theory:  
 instead, we deduce Corollary \ref{character var} (as well as Theorem \ref{main thm}) via a more direct and concrete method that invokes only basic facts of ordinary mixed Hodge theory as in \cite{Deligne}.\footnote{On the other hand, a major source of interest in twisted character varieties lies \cite{HR2} in non-abelian Hodge theory, specifically the $P=W$ conjecture.}.

Theorem \ref{main thm} has the following slightly different but equivalent formulation.  Choose a subgroup $\bS\subset \prod_i GL(\alpha_i)$ whose projection $\bS\rightarrow \mathbb{G}$ is a finite covering.  Then one can form the stack quotient 
$\on{Rep}(\Lambda^q,\alpha)^{\theta\on{-s}}/\bS$, which comes with a morphism $\pi: \on{Rep}(\Lambda^q,\alpha)^{\theta\on{-s}}/\bS\rightarrow  \mathcal{M}_{\theta}^q(\alpha)^{\on{s}}$ that is a gerbe, in fact a torsor over the commutative group stack $BH$ where $H=\ker(\mathbb{S}\rightarrow\mathbb{G})$.  We have an isomorphism 
$H^*(B\bS,\mathbb{Q})\cong H^*(B\mathbb{G},\mathbb{Q})$ and $\pi$ induces an isomorphism
$H^*\big(\mathcal{M}_{\theta}^q(\alpha)^{\on{s}},\mathbb{Q}\big)\cong H^*\big(\on{Rep}(\Lambda^q,\alpha)^{\theta\on{-s}}/\bS,\mathbb{Q}\big)$.  Thus Theorem \ref{main thm} can be restated as:
\begin{thm}\label{stack main thm}
For each connected open substack 
$U\subseteq \on{Rep}(\Lambda^q,\alpha)^{\theta\on{-s}}/\mathbb{S}$, 
the pure cohomology $PH^*(U)$ is generated as a $\mathbb{Q}$-algebra by the Chern classes of tautological bundles
$\on{Rep}(\Lambda^q,\alpha)^{\theta\on{-s}}\times_{\bS} V$ associated to finite-dimensional representations $V$ of $\bS$.  
\end{thm}
It is Theorem \ref{stack main thm} that we prove directly: the tautological bundles $\on{Rep}(\Lambda^q,\alpha)^{\theta\on{-s}}\times_{\bS} V$ that appear naturally and geometrically in our proof do not themselves descend to the multiplicative quiver variety in general, so it is more convenient to work  on the Deligne-Mumford stack $\on{Rep}(\Lambda^q,\alpha)^{\theta\on{-s}}/\bS$.

Unlike the situation of quiver varieties in \cite{McNKirwan}, we know of no obvious generalizations of Theorems \ref{main thm} and \ref{stack main thm} to other even-oriented cohomology theories (such as topological $K$-theory or elliptic cohomology).  However, we do obtain the following analogue of Theorem 1.6 of \cite{McNKirwan}.

\begin{thm}\label{derived cat}
Suppose there is some vertex $i\in I$ for which the dimension vector $\alpha$ satisfies $\alpha_i=1$, and let $\mathcal{M} = \mathcal{M}_{\theta}^q(\alpha)^{\on{s}}$. 
Let $D(\mathcal{M})$ denote the unbounded quasicoherent derived category of $\mathcal{M}$, and $D^b_{\on{coh}}(\mathcal{M})$ its bounded coherent subcategory.
\begin{enumerate}
\item The category $D(\mathcal{M})$ is generated by tautological bundles.  
\item There is a finite list of tautological bundles from which every object of  $D^b_{\on{coh}}(\mathcal{M})$ is obtained by finitely many applications of (i) direct sum, (ii) cohomological shift, and (iii) cone.
\end{enumerate}
\end{thm}
As for the analogous result in \cite{McNKirwan}, we emphasize that Theorem \ref{derived cat}(2) is {\em not} simply a formal consequence of Theorem \ref{derived cat}(1), since we do {\em not} include taking direct summands (i.e., retracts) among the operations (i)-(iii).  It would be interesting to know generators for $D^b_{\on{coh}}(\mathcal{M})$ for more general dimension vectors $\alpha$ than in Theorem \ref{derived cat}.

\subsection{Method of Proof}
The proof of Theorem \ref{stack main thm} is broadly similar to the proof used in \cite{McNKirwan} to establish that tautological classes generate the cohomology of Nakajima quiver varieties.  

A main part of the proof consists in producing a suitable modular compactification of the multiplicative quiver variety (or rather its Deligne-Mumford stack analogue).  One major difference from the Nakajima quiver variety case arises already at this stage: one frequently relies on $q$ being an appropriate tuple of primitive roots of unity to deduce that $\mathcal{M}_\theta^q(\alpha)$ parameterizes only stable representations, independently of the choice of $\theta$; whereas in \cite{McNKirwan}, we assumed, without significant loss of generality, that $\theta$ was a generic stability condition.  We note that such a genericity assumption here would exclude the possibility of applications to the character variety $\on{Char}(\Sigma_g, GL_n, q\on{Id})$; hence we avoid it.  Instead we identify a compactification by a ``projective Artin stack'' $\overline{\mathcal{M}}$, a quotient of a quasiprojective scheme by a reductive group whose coarse moduli space is a projective scheme.  Known techniques \cite{Kirwan, Edidin} allow us to replace the Artin stack compactification by a projective Deligne-Mumford stack at no cost to the validity of our approach.  

The second stage is to identify a complex on $\mathcal{M}_\theta^q(\alpha)\times\overline{\mathcal{M}}$ that, roughly speaking, resolves the graph of the embedding
$\mathcal{M}_\theta^q(\alpha)\hookrightarrow \overline{\mathcal{M}}$.  Again, while this is morally similar to \cite{McNKirwan}, the actual construction and proofs are more complicated and subtle.  This is essentially because our compactification of the Nakajima quiver variety relied on a graded 3-Calabi-Yau algebra, whereas the compactification of $\mathcal{M}_\theta^q(\alpha)$ uses an algebra, denoted by $\sA$ in the body of this paper, that may (conjecturally) be what one might call a ``relative $2g$-Koszul algebra'' in most cases but (as far as we know) is not known to be so.  Fortunately it turns out that we can proceed as if the algebra $\sA$ were known to have certain desired properties, carry out some constructions, and check by hand that the resulting complex behaves as hoped.  Unfortunately, in the generality in which we work here (and again unlike \cite{McNKirwan}), it seems one cannot expect the complex to actually provide a resolution of the structure sheaf of the graph of the embedding: instead, we rely on work of Markman \cite{Markman} to show that an appropriate Chern class of the complex we built is the Poincar\'e dual of the fundamental class of the graph. 

The final step is to deduce the theorem via usual integral transform arguments.  In \cite{McNKirwan}, we used Nakajima's result that the (integral) cohomology of a quiver variety is generated by algebraic cycles, hence is surjected onto by the cohomology of any compactification.  Such an assertion is not true of the multiplicative quiver varieties $\mathcal{M}_\theta^q(\alpha)$.  Instead, what is always true is that the cohomology of any reasonable smooth compactification---which is always Hodge-theoretically pure---surjects onto the pure part of the cohomology of any open subset.  This yields the assertion of the theorem, which in any case would be the best possible result, given that the cohomology $H^*(B\mathbb{G},\mathbb{Q})$ is pure; but its Hodge-theoretic nature also necessitates working with rational cohomology. 

 It is an interesting question to characterize the image of $H^*(B\mathbb{G},\mathbb{Z})$ in $H^*\big(\mathcal{M}_\theta^q(\alpha),\mathbb{Z}\big)$.  

\subsection{Acknowledgments}
We are grateful to Gwyn Bellamy, Ben Davison, Tam\'as Hausel, and Travis Schedler for helpful conversations, and to Donu Arapura and Ajneet Dhillon for help with references.
The first author was supported by EPSRC programme grant EI/I033343/1 and a Fisher Visiting Professorship at the University of Illinois at Urbana-Champaign.  The second author was supported by NSF grants DMS-1502125 and DMS-1802094 and a Simons Foundation fellowship.   The authors are also grateful to the Department of Mathematics of the University of Notre Dame for its hospitality 
during part of the preparation of this paper.

\subsection{Notation}
Throughout, $k$ denotes a field of characteristic $0$.  In Sections \ref{sec:intro} and \ref{sec:coh}, $k=\mathbb{C}$.

\section{Quivers and Multiplicative Preprojective Algebras}
\subsection{Truncations of Graded Algebras}
We will frequently use certain ``truncations'' of a $\mathbb{Z}_{\geq 0}$-graded algebra $A$ in what follows.  For a $\mathbb{Z}$-graded vector space $V$ and integer $n$, we write $V_{\geq n} = \oplus_{m\geq n} V_m$, a vector space graded by $\{n, n+1, \dots\}$.  We note the vector space injection $V_{\geq n}\rightarrow V$ that is the identity on the $m$th graded piece for $m\geq n$.  
\begin{defn}
For a 
$\mathbb{Z}_{\geq 0}$-graded algebra $A$ and each $N\geq 0$, we define:
$A\intN := A/A_{\geq N+1}.$
\end{defn}

\subsection{Quivers, Doubles, and Triples}
Let $Q = (I,\Omega)$ be a finite quiver, so that $s, t: \Omega\rightrightarrows I$ are the source and target maps: for $a\in\Omega$ we have $\xymatrix{\overset{s(a)}{\bullet}\ar[r]^{a} & \overset{t(a)}{\bullet}}$. 

The {\em double} of $Q$ is a quiver $\Qdbl = (I,H = \Omega\sqcup\overline{\Omega})$ with the same vertex set $I$ as for $Q$ and the set of arrows
$H = \Omega\sqcup\overline{\Omega}$ where $\Omega$ is the arrow set of $Q$ and $\overline{\Omega}$ is a set equipped with a bijection to $\Omega$, written $\Omega\ni a \leftrightarrow a^*\in \overline{\Omega}$.  We extend this bijection canonically to an involution on $H = \Omega\sqcup\overline{\Omega}$, still written $a\mapsto a^*$, and decree $s(a^*) = t(a), t(a^*)=s(a)$.  
 For each arrow $a\in H$ we write
\bd
\epsilon(a) = \begin{cases} 1 & \text{if}\hspace{.5em} a\in\Omega,\\ -1 &  \text{if}\hspace{.5em} a\in\overline{\Omega}.\end{cases}
\ed

Fix an integer $N\geq 1$.  The {\em graded tripled quiver} $Q^{\on{gtr}}$ associated to $Q$ (cf. Section 4 of \cite{McNKirwan}) is a quiver defined as follows.
We give $\Qgtr$ the vertex set $I^{\on{gtr}} = I\times [0, N]$ where $I$ is the vertex set of $Q$.  If $\Omega$ is the edge set of $Q$ and $H = \Omega\sqcup \overline{\Omega}$ the associated set of pairs of an edge together with an orientation, we give $\Qgtr$ the arrow set
\bd
\big(H\times [0, N-1]\big) \sqcup \big(I\times [0, N-1]).   \hspace{3em} \text{Thus,}
\ed
\begin{enumerate}
\item for each $h\in H$, $n\in [0, N-1]$  we have arrows $(h,n)$ with $\xymatrix{\overset{(s(h),n)}{\bullet} \ar[r]^{(h,n)} & \overset{(t(h),n+1)}{\bullet}}$, i.e.
\bd
s(h,n) = (s(h), n) \hspace{1em} \text{and} \hspace{1em} t(h,n) = (t(h), n+1);
\ed
\item for each $i\in I$, $n\in [0, N-1]$ we have arrows $t_{(i,n)}$  with $\xymatrix{\overset{(i,n)}{\bullet} \ar[r]^{t_{(i,n)}} & \overset{(i,n+1)}{\bullet}}$, i.e.
\bd
s(t_{i,n}) = (i,n) \hspace{1em} \text{and} \hspace{1em} t(t_{(i,n)}) = (i,n+1).
\ed
\end{enumerate}
More discussion can be found in \cite{McNKirwan}.  

\subsection{Path Algebras}\label{sec:path algebras}
Let $S = \bigoplus_i S e_i$ be a semisimple algebra with orthogonal system of idempotents $\{e_i\}$.  Suppose $A$ is an algebra with homomorphism $S\rightarrow A$.  
We say that $x\in A$ {\em has diagonal Peirce decomposition} if $\ds x\in \bigoplus_{i\in I} e_iAe_i$, or equivalently if it lies in the centralizer $Z_{A}(S)$.

Given a quiver $Q$,
We let $kQ$ denote the path algebra of the quiver.   Thus, we have a finite-dimensional semisimple $k$-algebra $S = \bigoplus_{i\in I} ke_i$ with idempotents $e_i$ labelled by the vertices $i\in I$.  We define an $S$-bimodule $B = B(Q)$, with $k$-basis labelled by the arrows, and ``arrows written left-to-right,'' so $e_iae_j =0$ unless $i = s(a), j=t(a)$, and so that $e_{s(a)}ae_{t(a)} = a$.  Then $kQ = T_S(B(Q))$ (the tensor algebra).

It is natural to grade the path algebra $kQ$ of any quiver $Q = (I,H)$---for example, $k\Qdbl$---by taking the semisimple algebra $S$ to lie in degree $0$ and the arrows $h\in H$ to lie in degree $1$: this is the standard nonnegative grading on the tensor algebra.  The algebra $k\Qdbl\langle t\rangle$ thus is naturally bi-graded, hence has total grading with $\deg(t)=1$.  
We can also grade $k\Qgtr$ by putting the semisimple algebra $\ds\bigoplus_{i\in I^{\on{gtr}}} ke_i$
 in degree $0$ and the arrows in degree $1$.  We obtain a graded algebra homomorphism
\bd
k\Qdbl\langle t\rangle\longrightarrow k\Qgtr \hspace{3em} \text{by taking }
\ed
\bd
e_i\mapsto \sum_n e_{(i,n)}, \;\; i\in I, 
\hspace{2em}
h\mapsto \sum_n (h,n), \;\; h\in H,
\hspace{2em}
t\mapsto \sum_{(i,n)}t_{(i,n)}.
\ed
The graded algebra $k\Qgtr$ has the property $k\Qgtr_{\geq N+1} = 0$, so we obtain a homomorphism
\begin{equation}\label{dbl to gtr}
k\Qdbl\langle t\rangle\intN := k\Qdbl\langle t\rangle/k\Qdbl\langle t\rangle_{\geq N+1}\longrightarrow k\Qgtr.
\end{equation}
\begin{lemma}\label{lem:ind of quiver reps}
Let $k\Qgtr\on{-mod}$ denote the category of finite-dimensional left modules, and furthermore 
let $k\Qdbl\langle t\rangle_N\on{-gr}_{[0,N]}$ denote the category of finite-dimensional graded left modules concentrated in degrees $[0,N]$.  
Then the homomorphism \eqref{dbl to gtr} determines an equivalence of categories:
\bd
k\Qgtr\on{-mod} \xrightarrow{\simeq} k\Qdbl\langle t\rangle\intN\on{-gr}_{[0,N]}.
\ed
\end{lemma}
\noindent
This equivalence identifies representations of $k\Qdbl[t]\intN$ with representations of the quotient $k\Qgtr/J$, where $J$ denotes the two-sided ideal
\begin{equation}\label{eq:t commutes}
J = \Big(\big\{ t_{(s(h), n)}\cdot(h, n+1) - (h,n)\cdot t_{(t(h),n+1)}\; \big| \; h\in H, \hspace{.3em} n\in [0,N-2]\big\}\Big).
\end{equation}

\subsection{Universal Localizations}
We briefly review some aspects of universal localizations that may be unfamiliar to the reader, using Chapter 4 of \cite{Schofield} as our reference; see also \cite{Cohn}.

Suppose that $R$ is a ring with $1$ and $\Sigma$ is a set of elements of $R$.  Then there is a ring $R_\Sigma$ with a homomorphism $R\rightarrow R_\Sigma$ that is universal with respect to the property that for every $r\in \Sigma$, $r$ becomes invertible in $R_\Sigma$.  The ring $R_\Sigma$ is called the {\em universal localization} of $R$ at $\Sigma$; an alternative notation that is sometimes preferable is $\Sigma^{-1}R$.   The universal localization is constructed as follows: letting $\Sigma^{-1}$ denote the set of symbols $a^{-1}$ for $a\in \Sigma$, we define
\bd
\Sigma\inv R = R_\Sigma := R\langle\Sigma\inv\rangle/(\{a\inv a -1 \; |\; a\in\Sigma\}).
\ed
This has the universal property claimed.  
We will need the following properties, which follow immediately from the universal property.
\begin{prop}\label{prop:univ loc}
Suppose $R$ is a ring with $1$.  
\begin{enumerate}
\item If $t\in Z(R)$ is central, then $R_{\{t\}}$ is isomorphic to the \O re localization of $R$ at $t$.  
\item If $\Sigma, \Sigma'\subseteq R$ are subsets, let $\overline{\Sigma}'$ denote the image of $\Sigma'$ in $R_\Sigma$.  Then $(R_\Sigma)_{\overline{\Sigma}'} \cong R_{\Sigma\cup\Sigma'}$.
\item Given a two-sided ideal $I\subseteq R$, let $\overline{\Sigma}$ denote the image of $\Sigma$ in $R/I$ and $I_\Sigma$ denote the two-sided ideal in $R_\Sigma$ generated by $I$.  Then $(R/I)_{\overline{\Sigma}} \cong R_\Sigma/I_\Sigma$.
\end{enumerate}
\end{prop}

\subsection{Multiplicative Preprojective Algebras} 
We review the multiplicative preprojective algebra of a quiver $Q$ as defined in \cite{CBS}.

Given a quiver $Q$ with double $\Qdbl = (I,H)$, for each arrow $a\in H$ of $\Qdbl$, we define $g_a = 1 + aa^*\in k\Qdbl$.  Write $L_Q$ for the algebra obtained by universal localization of $k\Qdbl$ inverting $\Sigma = \{g_a\; |\; a\in H\}$. Identify the tuple $q\in (k^\times)^I$ with the element $\sum_{i\in I} q_i e_i\in S$.  Crawley-Boevey and Shaw choose an ordering of the arrows in $H$ and define
$\ds \rho_{\on{CBS}} = \prod^{\longrightarrow}_{a\in H} g_a^{\epsilon(a)} - q$ (the arrow over the product indicates that it is taken in the chosen order).  It is proven in \cite{CBS} that, up to isomorphism, the quotient algebra $L_Q/(\rho_{\on{CBS}})$ does not depend on the choice of ordering.  Thus, in this paper we 
specifically fix an ordering $\Omega = \{a_1, \dots, a_g\}$ on the arrows in $Q$, and let
\begin{equation}\label{rho}
\rho_{\on{CBS}}  = g_{a_1}g_{a_2}\dots g_{a_g} g_{a_1^*}^{-1} \dots g_{a_g^*}^{-1}
-q.
\end{equation}
\begin{defn}
The associated multiplicative preprojective algebra is
\bd
\Lambda^q = \Lambda^q(Q) = L_Q/(\rho_{\on{CBS}}),
\ed
where $\rho_{\on{CBS}}$ is defined as in \eqref{rho}.
\end{defn}

\subsection{Homogenized Multiplicative Preprojective Algebras}
A principal tool in this paper is a certain graded algebra $\sA$ that ``homogenizes'' the multiplicative preprojective algebra $\Lambda^q$ of \cite{CBS}.  Here we construct the algebra $\sA$ and  collect some basic facts about $\sA$ and its relation to the multiplicative preprojective algebra $\Lambda^q$.

Thus, fix a quiver $Q$.  We consider $k\Qdbl[t] = k\Qdbl\langle t\rangle/(ta-at \; | \; a\in k\Qdbl)$ as a nonnegatively graded algebra, with the generators $a\in H, t$ all in degree 1, and $S = \oplus_{i\in I} ke_i$ in degree $0$.
We let 
\bd
G_a := t^2 + a a^* \in k\Qdbl[t] \hspace{2em} \text{for all} \hspace{2em} a\in H.
\ed
 \begin{remark}
 Each $G_a$ has diagonal Peirce decomposition: more precisely,
 \bd
 e_{s(a)}G_a = e_{s(a)}t^2 + aa^* =  G_a e_{s(a)}, \hspace{2em} \text{and} \hspace{2em} e_iG_a = e_it^2 = t^2 e_i = G_a e_i \hspace{1em} \text{for $i\neq s(a)$}.
 \ed
 \end{remark} 
 We note the obvious equalities
 \begin{equation}\label{G_a composed a}
 G_a a = aG_{a^*}, \hspace{3em} a^* G_a = G_{a^*} a^*.
 \end{equation}
\noindent
 Given $q \in (k^\times)^I$, we identify $q$ with $q:= \sum_{i\in I} q_ie_i\in k\Qdbl$, a sum of idempotents in the path algebra (which thus also has diagonal Peirce decomposition).  

Analogously to \cite{CBS}, the algebra $k\Qdbl[t]$ admits a universal localization in which the elements $G_a$, $a\in H$, and $t$ are inverted: we write $\sL_t$ for this universal localization.
The algebra $\sL_t$ contains invertible elements $\ds g_a = t^{-2}G_a = 1 + \frac{a}{t}\frac{a^*}{t}$ in graded degree $0$.  We have
$(\sL_t)_0 \cong L_Q$, where $L_Q$ is the universal localization of $k\Qdbl \cong k\Qdbl[t^{\pm 1}]_0$ at the elements $g_a$, $a\in H$, as in \cite{CBS} and reviewed above.
As above, fix an ordering $\Omega = \{a_1, \dots, a_g\}$ on the arrows in $Q$.  Write
\begin{equation}\label{D defn}
D = G_{a_1}\dots G_{a_g}, \hspace{1em} D^*  = q(G_{a_g^*}\dots G_{a_1^*}),
\end{equation}
\begin{equation}\label{rho defn}
\rho = D - D^* 
 = (G_{a_1}\dots G_{a_g}) - q(G_{a_g^*}\dots G_{a_1^*}) \in k\Qdbl[t].
\end{equation}

\begin{defn}
We write $\sA = k\Qdbl[t]/(\rho)$, where $(\rho)$ denotes the two-sided ideal generated by $\rho$.  
\end{defn}
The element $\rho$ has diagonal Peirce decomposition, and so $\rho e_i = e_i \rho$, and 
$(\rho) = (\{\rho e_i| i\in I\})$.
\begin{prop}\label{prop:algebra A}
Write $\Sigma = \{G_a \; | \; a\in H\}\cup\{t\}$.  We have:
\begin{enumerate}
\item $\sA$ is a graded algebra where $a_i, a_i^*$ and $t$ have degree $1$ (and $S = \sum_{i\in I} ke_i$ lies in degree $0$). 
\item The universal localization 
\begin{equation}\label{universal localization}
\Lambda_t:= \Sigma\inv\sA
\end{equation}
 of $\sA$ obtained by inverting all $G_a, a\in H$, and $t$, is a graded algebra, and $(\Lambda_t) \cong \Lambda^q(Q)[t^{\pm 1}]$ where $\Lambda^q(Q)=: \Lambda^q$ denotes the multiplicative preprojective algebra of \cite{CBS}.
\end{enumerate}
\end{prop}
The isomorphism \eqref{universal localization} of part (2) of Proposition \ref{prop:algebra A} follows from Proposition \ref{prop:univ loc}.

\section{Representations and their Moduli}
\subsection{Representations of $k\Qdbl$ and $k\Qgtr$}\label{reps of kQdbl and kQgtr}
Fixing some $N\geq 2g$, where $g$ is the number of arrows in $Q$, we form the graded-tripled quiver $\Qgtr$ associated to $Q$ as above.\footnote{Thus, in particular, $N$ is at least as large as the degree of the relation $\rho$.}  
Given a dimension vector $\alpha \in \mathbb{Z}_{\geq 0}^I$ for the quiver $\Qdbl$, we write $\algtr \in \mathbb{Z}_{\geq 0}^{I\times [0,N]}$ for the dimension vector for $k\Qgtr$ for which $\algtr_{i,n} = \alpha_i$ for all $n\in [0,N]$.  
We write $\on{Rep}(k\Qdbl,\alpha)$ for the space of representations of $k\Qdbl$ with dimension vector $\alpha$ and $\mathbb{G} = \prod_i GL(\alpha_i)$ for the automorphism group; thus
\bd
\on{Rep}(k\Qdbl,\alpha) := \prod_{h\in H} \on{Hom}(k^{\alpha_{s(h)}}, k^{\alpha_{t(h)}}).
\ed
 Similarly we write
$\on{Rep}(k\Qgtr,\algtr)$ for the space of representations of $k\Qgtr$ with dimension vector $\algtr$, and $\Ggtr$ for the automorphism group.

As in the construction of Section 4.3 of \cite{McNKirwan}, there is a natural ``induction functor'' from the category of representations of $k\Qdbl$ with dimension vector $\alpha$ to the category of representations of $k\Qgtr$ of dimension vector $\algtr$.  
The construction proceeds as follows.  To a representation $V$ of $k\Qdbl$ we may associate the $\mathbb{Z}_{\geq 0}$-graded vector space $V[t]$, and let arrows $h$ of $\Qdbl$ act as multiplication followed by shift-of-grading.  This makes $V[t]$ into a graded left $k\Qdbl[t]$-module.  We then form $V[t]/V[t]_{\geq N+1}$, a graded left $k\Qdbl[t]\intN$-module, and finally apply Lemma \ref{lem:ind of quiver reps} to get a representation of $k\Qgtr$: in fact, a representation of the quotient
$k\Qgtr/J$ where $J$ is as in \eqref{eq:t commutes}.

More concretely, the above construction is the following.
Suppose we have a representation $V = (V_i)_{i \in I}$ of $k\Qdbl$ of dimension vector $\alpha$.  
We obtain a representation  of $k\Qdbl[t]$ on a vector space $V_{\bullet, \bullet}$ of dimension vector $\algtr$ defined by:
\begin{enumerate}
\item setting $V_{i,n} := V_i$ for all $n\in [a,b]$;
\item  defining $t_{i,n} = t\cdot -: V_{i,n} = V_i \xrightarrow{\on{id}} V_i = V_{i,n+1}$ to act by shift of $\mathbb{Z}$-grading; and 
\item defining each generator of $k\Qdbl[t]$ corresponding to $h\in H$ to act as the composite
\bd
(h,n): V_{i,n} = V_i \xrightarrow{h\cdot -} V_i = V_{i,n}\xrightarrow{t\cdot -} V_{i,n+1}.
\ed  
\end{enumerate}

The construction determines a morphism of algebraic varieties (``induction'')
\bd
\mathsf{Ind}^\circ: \on{Rep}(k\Qdbl,\alpha)\longrightarrow \on{Rep}(k\Qgtr, \algtr).
\ed

\bd
\text{Write } \hspace{2em} 
 \mathbb{G} = \prod_i GL(V_{i})\hspace{1em}\text{and}\hspace{1em}  \Ggtr = \prod_{(i,n)\in I\times[a,b]} GL(V_{i,n}) \cong \prod_{n\in [a,b]}\mathbb{G},
\ed with the
diagonal homomorphism $\ds \on{diag}: \mathbb{G}\rightarrow \Ggtr \cong \prod_{n\in [a,b]} \mathbb{G}$.  Then the morphism $\mathsf{Ind}^\circ$ is $(\mathbb{G}, \Ggtr)$-equivariant.  We thus get a natural $\Ggtr$-equivariant morphism
\begin{equation}\label{Ind-map}
\mathsf{Ind}: \Ggtr\times_{\mathbb{G}} \on{Rep}(k\Qdbl,\alpha)\longrightarrow \on{Rep}(k\Qgtr, \algtr).
\end{equation}
Thus, given a representation $(a_h: V_{s(h)}\rightarrow V_{t(h)})_{h \in H}$ of $k\Qdbl$ on $V$, and $(g_{i,n})\in \Ggtr$, we have
\bd
\mathsf{Ind}\big((g_{i,n}), a_h\big) = \big((h,n), t_{i,n}\big) \, \text{where} \,  (h,n) = g_{t(h),n+1}a_h g_{s(h),n}\inv \, \text{and} \, 
t_{i,n} = g_{i,n+1} g_{i,n}\inv.
\ed

\begin{prop}
The map $\mathsf{Ind}$ of \eqref{Ind-map} defines a $\Ggtr$-equivariant open immersion of $\Ggtr\times_{\mathbb{G}} \on{Rep}(k\Qdbl,\alpha)$ in $\on{Rep}(k\Qgtr/J,\algtr)$, whose image consists of those $\big((h,n), t_{i,n}\big)$ for which:
\begin{equation}
\text{$t_{i,n}$ is an isomorphism for all $n\in [0,N-1]$.} \tag{$\dagger$}
\end{equation}
\end{prop}

\subsection{Representations of $\sA$ and $\sA\intN$}\label{induction and truncation}
Let $\sA\on{-Gr}$ denote the category of graded left $\sA$-modules.  
We also consider the category  $\sA\intN\on{-Gr}_{\geq 0}$ of  those graded left $\sA\intN$-modules $M$ for which $M_i=0$ for $i\notin [0,N]$.  We remark that $\sA\intN\on{-Gr}_{\geq 0}$ can naturally be viewed as a full subcategory of the category $\sA\intN\on{-Gr}$ of all graded left $\sA\intN$-modules, hence also of $\sA\on{-Gr}$.
Define a functor of truncation, 
\bd
\tau_{[0,N]}: \sA\on{-Gr} \longrightarrow \sA\intN\on{-Gr}_{\geq 0},
\ed by $M\mapsto \tau_{[0,N]}M := M_{\geq 0}/M_{\geq N+1}$.  As above, we have a graded vector space injection $\tau_{[0,N]}(M)\rightarrow M$ that is the identity on the $m$th graded piece for $m\in [0,N]$ and is zero elsewhere; this map is $\sA_{\leq m}$-linear on $M_{N-m}$.

\subsection{Representations of $\sA$ and $\Lambda^q$}
We note:
\begin{remark}
The functor $\Lambda_t\on{-Gr}\longrightarrow \Lambda^q\on{-Mod}$, $M\mapsto M_0$, is an equivalence of categories.
\end{remark}
Recall from \eqref{universal localization} that, letting $\Sigma = \{G_a \; | \; a\in H\}\cup\{t\}$, we have a graded algebra isomorphism $\Sigma\inv\sA \cong \Lambda_t = \Lambda^q[t^{\pm 1}]$, and hence a graded algebra homomorphism
$\sA\rightarrow \Lambda^q[t^{\pm 1}]$.
Given a left or right $\Lambda^q$-module $\overline{\mathcal{M}}$, we form a graded left or right $\Lambda_t$-module $\mathcal{M} = \overline{\mathcal{M}}[t^{\pm 1}]$, and thus a graded $\sA$-module $M = \mathcal{M}[t]=\mathcal{M}[t^{\pm 1}]_{\geq 0}$.   This defines a functor
\bd
R: \Lambda^q\on{-Mod}\longrightarrow \sA\on{-Gr}_{\geq 0}.
\ed
In the opposite direction, we have a functor $(\Lambda_t\otimes_{\sA} -)_0: \sA\on{-Gr}_{\geq 0}\rightarrow \Lambda^q\on{-Mod}$.
We have:
\begin{lemma}\label{dual of fin diml lambda module}
\mbox{}
\begin{enumerate}
\item The functors
$\xymatrix{(\Lambda_t\otimes_{\sA} -)_0: \sA\on{-Gr}_{\geq 0}\ar@<.5ex>[r] &  \Lambda^q\on{-Mod}:R\ar@<.5ex>[l]
}$
form an adjoint pair.
\item
If $\overline{\mathcal{M}}$ is a finite-dimensional left $\Lambda^q$-module then the graded left $\sA$-module $M = \overline{\mathcal{M}}[t]$ is finitely generated and projective as a left $S_t$-module and as a left $k[t]$-module.  Moreover, we have
$\Hom_{k[t]}(M,k[t]) \cong \Hom_k(\overline{\mathcal{M}},k)[t]$ as a graded right $\sA$-module.  
\end{enumerate}
\end{lemma}

\subsection{Representation Spaces and Group Actions}
Because the multiplicative preprojective algebra $\Lambda^q$ is the quotient $L_Q/(\rho_{\on{CBS}})$ of the localization $L_Q$ of $k\Qdbl$ by the ideal generated by $\rho_{\on{CBS}}$, the space $\on{Rep}(\Lambda^q,\alpha)$ of left $\Lambda^q$-modules with dimension vector $\alpha$ is naturally a locally closed subscheme of $\on{Rep}(k\Qdbl,\alpha)$: it is the closed subset, defined by vanishing of $\rho_{\on{CBS}}$, of the open set defined by invertibility of the elements $g_a$.

Similarly, the algebra $\sA\intN$ is a quotient of $k\Qdbl[t]\intN$ and thus, via Lemma \ref{lem:ind of quiver reps}, the space $\on{Rep}_{\on{gr}}(\sA\intN, \algtr)$ of graded left $\sA\intN$-modules concentrated in degrees $[0,N]$ is identified with a closed subscheme of $\on{Rep}(k\Qgtr, \algtr)$ defined by the vanishing of the images of $\rho$ and $J$ in $k\Qgtr$.  

It is immediate from the construction of Section \ref{reps of kQdbl and kQgtr}
that:
\begin{prop}[cf. Prop. 4.7 of \cite{McNKirwan}]
The morphism $\mathsf{Ind}$ of \eqref{Ind-map} restricts to an open immersion:
\bd
\mathsf{Ind}: \Ggtr \times_{\mathbb{G}} \on{Rep}(\Lambda^q,\alpha) \rightarrow \on{Rep}_{\on{gr}}(\sA\intN, \algtr).
\ed
Its image consists of those representations on which the elements $t, G_a$ act invertibly whenever their domain and target lie in the range $[0,N]$.
\end{prop}
\begin{corollary}
The map $\mathsf{Ind}$ defines an open immersion of moduli stacks 
\bd
\on{Rep}(\Lambda^q,\alpha)/\mathbb{G}\rightarrow \on{Rep}_{\on{gr}}(\sA\intN, \algtr)/\Ggtr.
\ed
\end{corollary}

\subsection{Semistability and Stability}\label{sec:semistability}
We next discuss (semi)stability of representations and the corresponding GIT quotients.

For any quiver $Q = (I,\Omega)$ with dimension vector $\alpha\in \mathbb{Z}^I_{\geq 0}$, a GIT stability condition is given by 
$\theta\in\mathbb{Z}^I_{\geq 0}$ satisfying $\sum_i \theta_i\alpha_i = 0$.  The vector $\theta$ determines a character
$\chi_{\theta}: \prod_i GL(\alpha_i)\rightarrow\Gm$, $\chi(g_i)_{i\in I} = \prod_i \det(g_i)^{\theta_i}$, and the condition 
$\sum_i \theta_i\alpha_i = 0$ guarantees that the diagonal copy $\Delta(\Gm)$ of $\Gm$ in  $\prod_i GL(\alpha_i)$ lies in the kernel of $\chi$; we require this because $\Delta(\Gm)$ acts trivially on $\on{Rep}(Q,\alpha)$.  Given dimension vectors $\beta,\alpha$, we write $\beta <\alpha$ if $\beta\neq\alpha$ and $\beta_i\leq \alpha_i$ for all $i\in I$.

We now turn to stability conditions for the doubled and tripled quivers $\Qdbl$ and $\Qgtr$ for a fixed quiver $Q$.  Suppose $\theta$ is a stability condition for $\Qdbl$ and dimension vector $\alpha$.  We construct a stability condition $\thetagtr$ for $\Qgtr$ with dimension vector $\algtr$ as follows.  For a representation $M$ of $k\Qgtr$ of dimension vector $\algtr$, we write $\delta_{i,n}(M) := \dim(M_{i,n})$; we will write $\thetagtr$ as a linear combination of the $\delta_{i,n}$.  Also, we note that it suffices to construct a {\em rational} linear functional $\thetagtr$, since any positive integer multiple of $\thetagtr$ evidently defines the same stable and semistable loci. 
   We fix an ordering on the vertices of $Q$, 
    identifying
$I = \{1, \dots, r\}$.
    and  a positive integer 
$\ds T \gg 0.$
  We define:
\bd
\thetagtr :=  \sum_{i=1}^r T^i\big[\delta_{i,N}-\delta_{i,0}\big]
+ \sum_{i\in I} \theta_i \delta_{i,0}.
\ed

\begin{prop}
 Suppose $M = \mathsf{Ind}(N)$ for some representation $N$ of $k\Qdbl$ with dimension vector $\alpha$.  Then $M$ is 
semistable, respectively stable, with respect to $\thetagtr$ if and only if $N$ is semistable, respectively stable, with respect to $\theta$.
\end{prop}
\noindent
The proof is an easy adaptation of that of Proposition 4.12(4) 
of 
\cite{McNKirwan}.

We remark that the above construction does {\em not} match \cite{McNKirwan}: there we chose to construct a stability $\thetagtr$ for $\Qgtr$ that would be nondegenerate if $\theta$ was, whereas here we ignore this possible requirement.
While it would be possible to copy the construction of a stability $\thetagtr$ from \cite{McNKirwan} and prove analogues of the statements of \cite{McNKirwan}, there are cases important to multiplicative quiver varieties in which it is not possible to find a stability condition for $k\Qdbl$ that is nondegenerate in the sense used in \cite{McNKirwan}: for example, the case when $Q$ has a single vertex and loops based at that vertex, with dimension vector $\alpha = n>1$.  However, again for multiplicative quiver varieties, in some interesting cases the choice of the parameter $q$ can guarantee that every semistable representation of $\Lambda^q$ is automatically stable (though not for numerical reasons, as nondegeneracy guarantees).  Indeed, we say $q = (q_i)_{i\in I}\in (k^\times)^I$ is a {\em primitive $\alpha$th root of unity} if  $q^\alpha :=\prod q_i^{\alpha_i} = 1$ and $q^\beta\neq 1$ for all $0<\beta<\alpha$.  
We
have:
\begin{lemma}[\cite{CBS}, Lemma 1.5]
\mbox{}
\begin{enumerate}
\item Suppose that $M$ is a representation of $\Lambda^q$ with dimension vector $\alpha$.  Then $q^\alpha = 1$.  
\item In particular, if $q$ is a primitive $\alpha$th root of $1$, then every representation of $\Lambda^q$ of dimension vector $\alpha$ is $\theta$-stable for every $\theta$.
\end{enumerate}
\end{lemma}
For example, if $Q= (\{\ast\},E)$ where $E$ has $g$ loops at $\ast$, $\alpha = n$, and $q$ is a primitive $n$th root of $1$, then every representation of $\Lambda^q$ of dimension $n$ is stable for every $\theta$; the corresponding moduli space of representations of $\Lambda^q$ is the character variety $\on{Char}(\Sigma_g, GL_n, q\on{Id})$ of the introduction.  
\begin{remark}
It would be interesting to characterize those stability conditions $\thetagtr$ for $k\Qgtr$ with the property that there is a stability condition $\theta$ for $k\Qdbl$ so that if $M = \mathsf{Ind}(N)$ then $M$ is $\thetagtr$-(semi)stable if and only if $N$ is $\theta$-(semi)stable.
\end{remark}

\begin{notation}
We write 
\bd
\mathcal{M}_{\theta}^q(\alpha) := \on{Rep}(\Lambda^q,\alpha)/\!\!/_{\theta}\mathbb{G}
\;\;\text{and}\;\;
 \mathcal{M}_{\theta}^q(\alpha)^{\on{s}} := \on{Rep}(\Lambda^q,\alpha)^{\theta-\on{s}}/\!\!/_{\theta}\mathbb{G}
 \ed
 for the coarse moduli spaces determined by a stability condition $\theta$.
 \end{notation}

\subsection{Moduli Stacks and Resolutions}\label{sec:moduli stacks}
The moduli stacks 
\begin{center}
$\on{Rep}(\Lambda^q,\alpha)^{\theta\on{-ss}}/\mathbb{G}$ and $\on{Rep}_{\on{gr}}(\sA\intN, \algtr)^{\thetagtr\on{-ss}}/\Ggtr$ 
\end{center}
are never Deligne-Mumford stacks: the diagonal copy of $\Gm$ in $\mathbb{G}$, respectively $\Ggtr$, always acts trivially on $\on{Rep}(\Lambda^q,\alpha)$, respectively 
$\on{Rep}_{\on{gr}}(\sA\intN, \algtr)^{\thetagtr\on{-ss}}$.
Thus, the moduli stack of stable
representations $\on{Rep}(\Lambda^q,\alpha)^{\theta\on{-s}}/\mathbb{G}$ is always a $\Gm$-gerbe over the moduli space 
$\mathcal{M}_{\theta}^q(\alpha)^{\on{s}}$ of stable representations.  

However, one can make a choice of subgroup $\bS\subset \mathbb{G}$ that ensures that the quotient stack
$\on{Rep}(\Lambda^q,\alpha)^{\theta-\on{s}}/\bS$ is a Deligne-Mumford stack and that 
$\on{Rep}(\Lambda^q,\alpha)^{\theta-\on{s}}/\bS \rightarrow\mathcal{M}_{\theta}^q(\alpha)^{\on{s}}$ is a finite gerbe (indeed a principal $BH$-bundle for a finite abelian group $H$).  Indeed, for example, we can choose any character $\rho: \Ggtr\rightarrow \Gm$ for which the composite with the diagonal embedding $\rho\circ\Delta: \Gm\rightarrow \Gm$ is nontrivial, hence surjective.  Then $\Sgtr:=\ker(\rho)$ has the property that $\Ggtr = \Sgtr\cdot \Delta(\Gm)$ and similarly letting 
$\bS = \mathbb{G}\cap \Sgtr$ we have $\mathbb{G} = \bS\cdot\Delta(\Gm)$.  Moreover, since $\Delta(\Gm)$ is the stabilizer of every point of $\on{Rep}(\Lambda^q,\alpha)^{\theta\on{-s}}$ and $H := \Delta(\Gm)\cap \bS$ is finite, we get:
\begin{lemma}
The quotient $\on{Rep}(\Lambda^q,\alpha)^{\theta\on{-s}}/\bS$ is a Deligne-Mumford stack and the natural morphism
\bd
\on{Rep}(\Lambda^q,\alpha)^{\theta\on{-s}}/\bS\rightarrow  \mathcal{M}_{\theta}^q(\alpha)^{\on{s}}
\ed
is a torsor for the commutative group stack $BH$ (in particular, is a finite gerbe over $\mathcal{M}_{\theta}^q(\alpha)^{\on{s}}$).
\end{lemma}
By construction, we have an open immersion:
\bd
\on{Rep}(\Lambda^q,\alpha)^{\theta\on{-s}}/\bS\hookrightarrow \on{Rep}_{\on{gr}}(\sA\intN, \algtr)^{\thetagtr\on{-ss}}/\Sgtr,
\ed
and the coarse space of the target $\on{Rep}_{\on{gr}}(\sA\intN, \algtr)^{\thetagtr\on{-ss}}/\Sgtr$ is the projective moduli scheme $\on{Rep}_{\on{gr}}(\sA\intN, \algtr)/\!\!/_{\thetagtr}\Ggtr$: it is projective because it is a closed subscheme of 
$\on{Rep}(k\Qgtr,\algtr)/\!\!/_{\thetagtr}\Ggtr$, which (as in \cite{McNKirwan}) is itself projective because $k\Qgtr$ has no oriented cycles.

As in \cite{McNKirwan}, since our goal is to compactify $\on{Rep}(\Lambda^q,\alpha)^{\theta\on{-s}}/\bS$ appropriately, we will replace the quotient stack $\on{Rep}_{\on{gr}}(\sA\intN, \algtr)^{\thetagtr\on{-ss}}/\Sgtr$ by its closed substack defined as the closure of $\on{Rep}(\Lambda^q,\alpha)^{\theta\on{-s}}/\bS$.
\begin{notation}
We denote the closure of  $\on{Rep}(\Lambda^q,\alpha)^{\theta\on{-s}}/\bS$ in 
$\on{Rep}_{\on{gr}}(\sA\intN, \algtr)^{\thetagtr\on{-ss}}/\Sgtr$ by $\overline{\mathcal{M}}_{\on{st}}$.
\end{notation}
\begin{lemma}
The stack $\on{Rep}(\Lambda^q,\alpha)^{\theta\on{-s}}/\bS$ is smooth.  The stack $\overline{\mathcal{M}}_{\on{st}}$ is integral and its coarse moduli space is a projective scheme.  The natural morphism $\on{Rep}(\Lambda^q,\alpha)^{\theta\on{-s}}/\bS\hookrightarrow \overline{\mathcal{M}}_{\on{st}}$ is an open immersion.
\end{lemma}
\begin{proof}
The smoothness of $\on{Rep}(\Lambda^q,\alpha)^{\theta\on{-s}}/\bS$ is Theorem 1.10 of \cite{CBS}.  The remaining assertions are immediate.
\end{proof}
\noindent
We may apply the results of \cite{Kirwan} or \cite{Edidin} to $\on{Rep}(k\Qgtr, \algtr)^{\thetagtr\on{-ss}}/\Sgtr$ and its closed substack $\overline{\mathcal{M}}_{\on{st}}$ to obtain a projective Deligne-Mumford stack (i.e., a Deligne-Mumford stack whose coarse space is a projective scheme) $\overline{\mathcal{M}}_{\on{st}}'$
equipped with a projective morphism  $\overline{\mathcal{M}}_{\on{st}}'\rightarrow \overline{\mathcal{M}}_{\on{st}}$ that is an isomorphism over $\on{Rep}(\Lambda^q,\alpha)^{\theta\on{-s}}/\bS$.  The stack $\overline{\mathcal{M}}_{\on{st}}'$ is itself, by construction, a global quotient of a quasiprojective variety by $\bS$, and thus we may apply equivariant resolution to resolve the singularities of  $\overline{\mathcal{M}}_{\on{st}}'$, to obtain:
\begin{prop}\label{prop:compactification}
The smooth Deligne-Mumford stack $\on{Rep}(\Lambda^q,\alpha)^{\theta\on{-s}}/\bS$ admits an open immersion
\bd
\on{Rep}(\Lambda^q,\alpha)^{\theta\on{-s}}/\bS\hookrightarrow\overline{\on{Rep}(\Lambda^q,\alpha)^{\theta\on{-s}}/\bS}
\ed
in a smooth projective Deligne-Mumford stack equipped with a projective morphism $\overline{\on{Rep}(\Lambda^q,\alpha)^{\theta\on{-s}}/\bS}\rightarrow\overline{\mathcal{M}}_{\on{st}}$ that is compatible with the open immersion 
$\on{Rep}(\Lambda^q,\alpha)^{\theta\on{-s}}/\bS\hookrightarrow \overline{\mathcal{M}}_{\on{st}}$.
\end{prop}

\section{The Diagonal of the Algebra $\sA$}
\subsection{Bimodule of Derivations}
Recall that we have fixed an ordering $\Omega = \{a_1, \dots, a_g\}$ on the arrows in $Q$.
For $j=1, \dots, g$ we write
\bd
L_{a_j} = G_{a_1}\dots G_{a_{j-1}}, \hspace{1em} R_{a_j} = G_{a_{j+1}}\dots G_{a_g}, \hspace{1em}\text{so}\hspace{1em} D = L_{a_j}G_{a_j}R_{a_j},
\ed
\bd
L_{a_j^*} = G_{a_g^*}\dots G_{a_{j+1}^*}, \hspace{1em} R_{a_j^*} = G_{a_{j-1}^*}\dots G_{a_1^*}, \hspace{1em}\text{so}\hspace{1em} D^* = qL_{a_j^*}G_{a_j^*}R_{a_j^*}.
\ed

Let $B$ denote the sub-$(S[t],S[t])$-bimodule of $k\Qdbl[t]$ spanned by the arrows, so that $k\Qdbl[t]$ is identified with the tensor algebra $T_{S[t]}(B)$.  As in \cite[p.190]{CBS}, the bimodule that is the target of the universal $\sS$-linear bimodule derivation of $k\Qdbl[t]$ satisfies
\bd
\Omega_{S[t]}(k\Qdbl[t]) \cong k\Qdbl[t]\otimes_{S[t]} B\otimes_{S[t]} k\Qdbl[t],
\ed
under which the universal derivation $\delta_{k\Qdbl[t]/S[t]}: k\Qdbl[t] \rightarrow\Omega_{S[t]}(k\Qdbl[t])$ is identified with $a \mapsto 1\otimes a\otimes 1$.  As in \cite[p.~190]{CBS}, for the universal localization $\sL_t$ we also get
$\Omega_{S[t]}(\sL_t) \cong \sL_t \otimes_{k\Qdbl[t]} \Omega_{S[t]}(k\Qdbl[t]) \otimes_{k\Qdbl[t]} \sL_t$ with the obvious identification of the universal derivation $\delta_{\sL_t/S[t]}$.  We write:
\begin{equation}\label{P_1-eq}
P_1 = \sA\otimes_{S[t]} B\otimes_{S[t]} \sA \cong \sA \underset{k\Qdbl[t]}{\otimes} \Omega_{S[t]}(k\Qdbl[t]) \underset{k\Qdbl[t]}{\otimes} \sA.
\end{equation}
The module $P_1$ is evidently projective as a bimodule.  Via the above description, we obtain a collection of bimodule basis elements 
\bd
\eta_a, a\in H, \;\;\;  \text{via} \;\;\;
\eta_a = 1\otimes a \otimes 1 \in \sA\otimes_{S[t]} B\otimes_{S[t]} \sA = P_1.
\ed
  
\subsection{An Exact Sequence}
We write 
\bd
P_0 = \sA \otimes_{S[t]} \sA.
\ed
Write $\eta_i = e_i\otimes 1 = 1\otimes e_i$, $i\in I$, for the obvious bimodule generators of $P_0$. 
Define graded bimodule maps 
\begin{equation}\label{the-complex}
P_0(-2g) \xrightarrow{\alpha} P_1(-1) \xrightarrow{\beta} P_0
\end{equation}
by
$\beta(\eta_a) = a \eta_{s(a)} - \eta_{t(a)} a$ for arrows $a$ of $\Qgtr$, and
\begin{equation}\label{alpha formula}
\alpha(\eta_i) = \sum_{a\in\Omega, s(a) = i} L_a \Delta_a R_a -  \sum_{a\in\Omega, t(a)=i}  qL_{a^*}\Delta_{a^*}R_{a^*},
\end{equation}
 where $\Delta_a = \delta(G_a)$ (where $\delta$ denotes the universal derivation).  It is then immediate that $\alpha(\eta_i) = e_i\cdot\delta(\rho)$; in particular, letting $\theta: P_0\rightarrow (\rho)/(\rho^2)$ denote the map defined by $\theta(p\otimes q) = p\rho q$ and writing $\phi$ for the isomorphism defined by \eqref{P_1-eq}, we have:
 \begin{equation}\label{Omega equation}
  \phi\circ\alpha = \delta\circ\theta.
  \end{equation}
Imitating the proof of Lemma 3.1 of \cite{CBS} gives:
\begin{lemma}
The sequence
\begin{equation}\label{complex P}
P_0(-2g)\xrightarrow{\alpha} P_1(-1) \xrightarrow{\beta} P_0 \xrightarrow{\gamma} \sA\rightarrow 0,
\end{equation}
where $\gamma (p\otimes q) = pq$, is an exact sequence of $\mathbb{Z}$-graded bimodules.  
\end{lemma}
\begin{proof}
As in \cite[Theorem~10.3]{Schofield}, one gets an exact sequence
\bd
(\rho)/(\rho^2) \xrightarrow{\delta} \Omega_{S[t]}\big(k\Qdbl[t]\big) \rightarrow \Omega_{S[t]} \sA \rightarrow 0.
\ed
As in \cite{CBS}, splicing this sequence and the defining sequence for $\Omega_{S[t]}\big(\sA\big)$  and applying \eqref{Omega equation} gives a commutative diagram
\bd
\xymatrix{
P_0(-2g)\ar[d]^{\theta} \ar[r]^{\alpha} & P_1(-1) \ar[d]_{\cong}^{\phi} \ar[r]^{\beta} & P_0 \ar[d]_{\cong}^{\psi} \ar[r] & \sA \ar[d]^{=} \ar[r] & 0\\
(\rho)/(\rho^2) \ar[r]^{\hspace{-4.5em}\delta} &
\sA \underset{k\Qdbl[t]}{\otimes} \Omega_{S[t]}(k\Qdbl[t]) \underset{k\Qdbl[t]}{\otimes} \sA \ar[r]^{\hspace{4.5em}\xi} & 
\sA \otimes_{S[t]} \sA \ar[r] & \sA \ar[r] & 0.
}
\ed
The vertical arrows $\phi, \psi$ are isomorphisms and $\theta$ is surjective, yielding the assertion.
\end{proof}

\subsection{Dual of the Map $P_0(-2g)\xrightarrow{\alpha}  P_1(-1)$}
Recall that the {\em enveloping algebra} of $\sA$ over $k[t]$ is
\bd
\sA^e := \sA\otimes_{k[t]}\sA^{\on{op}}.
\ed
We consider $\sA^e$ as a left $\sA^e$-module
 where $a\otimes a'\in\sA^e$ acts by
\bd
a\otimes a'\cdot (x\otimes x') = ax\otimes x'a'.
\ed
We remark that $\sA^e$ naturally also has a {\em right} $\sA^e$-module structure commuting with the left $\sA^e$-action, where $a\otimes a'\in \sA^e$ acts on the right by
\bd
(x\otimes x')\cdot a\otimes a' = xa\otimes a'x'.  
\ed
Given a finitely generated left $\sA^e$-module, we form $P^\vee = \Hom_{\sA^e}(P, \sA^e)$, the dual over the enveloping algebra; by the above discussion, this module has a right $\sA^e$-module structure, which we can identify with a left $\sA^e$-module structure via the isomorphism
\bd
(\sA^e)^{\on{op}} \rightarrow \sA^e, \hspace{2em} a\otimes a' \mapsto a' \otimes a.
\ed

We now want to calculate the dual $\alpha^\vee$ of the map $\alpha$ of \eqref{the-complex} using the formula \eqref{alpha formula}.  Note that 
\bd
\Delta(G_a) = a\delta(a^*) + \delta(a)a^* = a \eta_{a^*} + \eta_a a^*.
\ed
We thus find from Formula \eqref{alpha formula} that the $\eta_a$-component of $\alpha$ is given by
\bd
\alpha(\eta_i)_{\eta_a} = 
\begin{cases} 
L_a \eta_a a^* R_a - qL_{a^*}a^*\eta_a R_{a^*} & \text{if $a\in\Omega$, $i=s(a)$},\\
L_{a^*} a^* \eta_a R_{a^*} - qL_a \eta_a a^* R_a & \text{if $a\in\overline{\Omega}$, $i=t(a)$}
\end{cases}
\ed
and zero otherwise.
Let $\{\eta_a^\vee\}$ denote the basis of $P_1^\vee$ dual to the basis $\{\eta_a\}$ of $P_1$; we note that 
\begin{equation}
\eta^\vee_a \in e_{t(a)}P_1^\vee e_{s(a)}.
\end{equation}
  It follows from the above formulas:
\begin{equation}\label{alpha check formulas}
\alpha^\vee(\eta_a^\vee) =
\begin{cases}
a^*R_a \eta_{s(a)}^\vee L_a - qR_{a^*} \eta_{t(a)}^\vee L_{a^*}a^* & \text{if $a\in\Omega$,}\\
R_a^*\eta_{t(a)}^\vee L_{a^*}a^* - qa^*R_a \eta_{s(a)}^\vee L_a & \text{if $a\in\overline{\Omega}$.}
\end{cases}
\end{equation}
\begin{lemma}\label{alpha vee formulas}
For all $a\in \Omega$, we have
\begin{align}
\alpha^\vee\big(\eta^\vee_a a- a^*\eta_{a^*}^\vee\big) & =  G_{a^*}\big(qR_{a^*}\eta^\vee_{t(a)} L_{a^*}\big) - \big(qR_{a^*}\eta^\vee_{t(a)}L_{a^*}\big)G_{a^*},\\
\alpha^\vee\big(a\eta_a^\vee - \eta_{a^*}^\vee a^*\big) & =  G_a\big(R_a\eta_{s(a)}^\vee L_a\big) - \big(R_a\eta_{s(a)}^\vee L_a\big) G_a.
\end{align}
\end{lemma}

\begin{lemma}\label{easy commuting case}
If $a\in H$, $s(a)\neq i$, then $G_a D\eta^\vee_i = D\eta_i^\vee G_a$ in $P_0^\vee$.
\end{lemma}
\begin{proof}
The element $D$ is a product of elements of diagonal Peirce type, hence itself is of diagonal Peirce type.  Thus, using $e_{s(a)}\eta_i^\vee = 0 = \eta_i^\vee e_{s(a)}$, we get
\begin{multline*}
G_a D\eta^\vee_i = \big(G_ae_{s(a)} + (1-e_{s(a)})t^2\big)D\eta^\vee_i = (1-e_{s(a)}t^2)D\eta^\vee_i \\
= D\eta^\vee_i(1-e_{s(a)}t^2) = D\eta^\vee_i\big(e_{s(a)}G_a + (1-e_{s(a)})t^2\big) = D\eta^\vee_i G_a.
\end{multline*}
This completes the proof.
\end{proof}
Suppose now that $\mathcal{M}$ is a graded right $\Lambda_t$-module; then $M = \mathcal{M}_{\geq 0}$ is a graded right $\sA$-submodule of $\mathcal{M}$.  For example, we could take $\mathcal{M} = \Lambda_t$ itself, as in \eqref{universal localization}.  
We consider the map 
\bd
M\otimes_{\sA} P_1^\vee(1) \xrightarrow{1_{M}\otimes\alpha^\vee} M\otimes_{\sA} P_0^\vee(2g).
\ed
\begin{remark}\label{defined operators}
We note that, under the above hypothesis on $M$, for any product $Q$ of elements $G_a$, $a\in H$, of degree $\deg(Q)$, the elements $Qt^{-\deg(Q)}$ and $t^{\deg(Q)}Q\inv$ of $\Lambda_t$ give well defined operators of right multiplication on $M$ that satisfy all relations in $\Lambda_t$. 
\end{remark}
\begin{prop}\label{big formulas}
Suppose that $M = \mathcal{M}_{\geq 0}$ for a graded right $\Lambda_t$-module $\mathcal{M}$.  
Then for all $m\in M$ and all $i\in I$ and $1\leq j \leq g$,
\begin{enumerate}
\item
 the elements
$m\big(G_{a_j} D \eta^\vee_i - D\eta_i^\vee G_{a_j}\big)$, $m\big(G_{a_j^*} D \eta^\vee_i - D\eta_i^\vee G_{a_j^*}\big)$, and 
\item
 the elements
$m\big(a_j^* Dt^{-2}\eta_{s(a_j)}^\vee - Dt^{-2}\eta_{t(a_j)}^\vee  a_j^*\big)$,  
$m\big(a_j Dt^{-2}\eta_{t(a_j)}^\vee - Dt^{-2}\eta_{s(a_j)}^\vee  a_j\big)$
\end{enumerate}
lie in $\on{Im}(1_{\Lambda_t}\otimes\alpha^\vee) \subseteq M\otimes_{\sA} P_0^\vee(2g)$.
\end{prop}
\begin{proof} 
(1) We first prove that $m\big(G_{a_j} D \eta^\vee_i - D\eta_i^\vee G_{a_j}\big) \in \on{Im}(1_M \otimes\alpha^\vee)$ by (strong) induction on $j$.  

\vspace{.5em}

\noindent 
{\em Base Case.} $j=1$.  By Lemma \ref{easy commuting case}, the assertion is true for $i\neq s(a_1)$.  From Lemma \ref{alpha vee formulas}, we have
\bd
mG_{a_1}\alpha^\vee\big(a_1\eta_{a_1}^\vee - \eta_{a_1^*}^\vee a_1^*\big) =
mG_{a_1}G_{a_1}\big(R_{a_1}\eta_{s(a_1)}^\vee L_{a_1}\big) - mG_{a_1}\big(R_{a_1}\eta_{s(a_1)}^\vee L_{a_1}\big) G_{a_1} 
= mG_{a_1}D\eta_{s(a_1)} - mD\eta_{s(a_1)}G_{a_1}.
\ed
This completes the base case.

\vspace{.5em}

\noindent 
{\em Induction Step.}  Assume $m\big(G_{a_k}D\eta^\vee_i - D\eta^\vee_i G_{a_k}\big)\in\on{Im}(1_{M}\otimes\alpha^\vee)$ for all $i\in I$ and $k<j$.  Again, by Lemma \ref{easy commuting case}, we have $mG_{a_j}D\eta^\vee_i - mD\eta^\vee_i G_{a_j}\in\on{Im}(1_M\otimes\alpha^\vee)$ for $i\neq s(a_j)$.  Applying Lemma \ref{alpha vee formulas} gives
\begin{multline*}
mG_{a_j}\alpha^\vee\big(a_j\eta_{a_j}^\vee - \eta_{a_j^*}^\vee a_j^*\big) =
mG_{a_j}G_{a_j}\big(R_{a_j}\eta_{s(a_j)}^\vee L_{a_j}\big) - mG_{a_j}\big(R_{a_j}\eta_{s(a_j)}^\vee L_{a_j}\big) G_{a_j} \\
= mG_{a_j}\big(t^{\deg(L_{a_j})}L_{a_j}\inv Dt^{-\deg(L_{a_j})}\eta_{s(a_j)}^\vee L_{a_j}\big) - \big(t^{\deg(L_{a_j})}L_{a_j}\inv Dt^{-\deg(L_{a_j})}\eta_{s(a_j)}^\vee L_{a_j}\big) G_{a_j}\\
= m\big(G_{a_j} D\eta_{s(a_j)}^\vee - D\eta_{s(a_j)}^\vee G_{a_j}\big),
\end{multline*}
where the last equality applies the inductive hypothesis.  This completes the induction step, thus proving the assertion for the elements $G_{a_j} D \eta^\vee_i - D\eta_i^\vee G_{a_j}$.

The proof for $G_{a_j^*} D \eta^\vee_i - D\eta_i^\vee G_{a_j^*}$ follows the analogous descending induction on $j$.

(2) Taking note of Remark \ref{defined operators}, from \eqref{alpha check formulas} we have 
$\alpha^\vee(mG_{a_j}t^{-2}\eta_{a_j}) = mG_{a_j}t^{-2}a_j^*R_{a_j} \eta_{s(a_j)}^\vee L_{a_j} - mG_{a_j}t^{-2}qR_{a_j^*} \eta_{t(a_j)}^\vee L_{a_j^*}a_j^*$.  Applying part (1) of the proposition to the right-hand side of this formula gives
\begin{align*}
\alpha^\vee(mG_{a_j}t^{-2}\eta_{a_j})  & = mG_{a_j}t^{-2}a_j^*R_{a_j}L_{a_j} \eta_{s(a_j)}^\vee - mqG_{a_j}t^{-2}R_{a_j^*} L_{a_j^*}\eta_{t(a_j)}^\vee a_j^*  + \on{Im}(1_{\Lambda_t}\otimes\alpha^\vee)\\
& = mG_{a_j}a_j^* G_{a_j}\inv Dt^{-2}\eta_{s(a_j)}^\vee - mG_{a_j}G_{a_j^*}\inv  Dt^{-2}\eta_{t(a_j)}^\vee  a_j^* + \on{Im}(1_{\Lambda_t}\otimes\alpha^\vee)
\\
& = mG_{a_j}G_{a_j^*}\inv\big(a_j^* Dt^{-2}\eta_{s(a_j)}^\vee - Dt^{-2}\eta_{t(a_j)}^\vee  a_j^* \big) + \on{Im}(1_{\Lambda_t}\otimes\alpha^\vee)
\end{align*}
where the last equality uses \eqref{G_a composed a}; in particular this gives the first assertion of Part (2) of the proposition.  The second assertion follows similarly.
\end{proof}

\section{Analysis of the Ext-Complex}

\subsection{The Complex \eqref{complex P} and the $\Hom$-Functor}
Let $M, N$ be graded left $\sA$-modules such that $M$ is finitely generated and projective as a $k[t]$-module.  To the exact sequence
\bd
P_0(-2g)\otimes_{\sA} M \xrightarrow{\alpha\otimes 1} P_1(-1)\otimes_{\sA} M \xrightarrow{\beta\otimes 1} P_0\otimes_{\sA} M\xrightarrow{\gamma\otimes 1} M\rightarrow 0
\ed
 we apply the functor $\Hom_{\sA}(-, N)$ to obtain an exact sequence
 \begin{equation}\label{Hom sequence beginning}
 0\rightarrow \Hom_{\sA}(M, N)\rightarrow \Hom_{\sA}(P_0\otimes_{\sA} M, N)
 \xrightarrow{(\beta\otimes 1)^*} \Hom_{\sA}(P_1(-1)\otimes_{\sA} M, N).
 \end{equation}
We continue the sequence \eqref{Hom sequence beginning} using
\begin{equation}\label{used alpha dual}
\Hom_{\sA}(P_1(-1)\otimes_{\sA} M, N)\xrightarrow{(\alpha\otimes 1)^*} \Hom_{\sA}(P_0(-2g)\otimes_{\sA} M, N).
\end{equation}
Thus, we would like to compute the cokernel of the map \eqref{used alpha dual}.  
\begin{prop}\label{Hom equals tensor}
Let $M, N$ be graded left $\sA$-modules such that $M$ is finitely generated and projective as a $k[t]$-module, and write $M^* = \Hom_{k[t]}(M,k[t])$.
 Consider the contravariant functors of finitely generated projective $\sA^e$-modules $P$, 
\bd
P\mapsto \big(N\otimes_{k[t]} M^*\big) \otimes_{\sA^e} P^\vee \hspace{2em} \text{and} \hspace{2em}
P\mapsto \Hom_{\sA}(P\otimes_{\sA} M, N).
\ed
The natural transformation 
$\big(N\otimes_{k[t]} M^*\big) \otimes_{\sA^e} P^\vee \xrightarrow{\Psi} \Hom_{\sA}(P\otimes_{\sA} M, N)$
of these functors of projective $\sA^e$-modules $P$ is a natural isomorphism.
\end{prop}
\begin{proof}
By projectivity, it suffices to check for $P=\sA^e$, where it follows by adjunction.
\end{proof}
\begin{corollary}
Under the hypotheses of Proposition \ref{Hom equals tensor}, the cokernel of the map \eqref{used alpha dual} is 
\bd
\on{coker}\big(1_{M^*}\otimes \alpha^\vee \otimes 1_N: M^*\otimes_{\sA} P_1^\vee(1)\otimes_{\sA} N\rightarrow  M^*\otimes_{\sA} P_0^\vee(2g)\otimes_{\sA} N\big).
\ed
\end{corollary}

We note the following identities, which are immediate from adjunction:
\begin{lemma}\label{adjunction identities}
Suppose that $M = \overline{M}[t]$ is the graded left $\sA$-module associated to a finite-dimensional left $\Lambda^q$-module $\overline{M}$.  Then:
\bd
\Hom_{\sA}(P_1\otimes_{\sA} M, N)\cong \Hom_{\sA}(\sA\otimes_{S_t} B[t]\otimes_{S_t} \sA\otimes_{\sA} M, N) \cong \Hom_{S_t}(B\otimes_S M, N)\cong \Hom_S(B\otimes_S \overline{M}, N),
\ed
\bd
\Hom_{\sA}(P_0\otimes_{\sA} M, N)\cong \Hom_{\sA}(\sA\otimes_{S_t}  \sA\otimes_{\sA} M, N) \cong \Hom_{S_t}(M, N)\cong \Hom_S(\overline{M}, N).
\ed
\end{lemma}

\subsection{The Ext-Complex}\label{sec:Ext complex}
Fix $N\geq 2g$.  Let $\overline{V}$ be a finite-dimensional representation of $\Lambda^q$ of dimension vector $\alpha$, and let $V=\overline{V}[t]$ be the corresponding graded $\sA$-module as in Section \ref{induction and truncation}, and specifically as in Lemma \ref{dual of fin diml lambda module}.
Suppose $W$ is a $\mathbb{Z}_{\geq 0}$-graded $\sA\intN = \sA/\sA_{\geq N+1}$-module, identified with a representation of $\Qgtr$ that has dimension vector $\algtr$.  Thus $\tau_{[0,N]}V$ is also identified with a representation of $\Qgtr$ that has dimension vector $\algtr$.

Let $P_\bullet$ denote the complex of \eqref{the-complex}.  We consider the complex $\Hom_{\sA}(P_\bullet\otimes_{\sA} V, W)$.  Since the sources and target of the Homs in this complex are graded $\sA$-modules, each Hom-space can be regarded as a graded vector space; we write
\bd
\mathsf{Ext} = \left[ \Hom_{\sA\on{-Gr}}(P_0\otimes_{\sA} V, W) \xrightarrow{\beta^\vee} \Hom_{\sA\on{-Gr}}(P_1\otimes_{\sA} V, W(1)) \xrightarrow{\alpha^\vee} \Hom_{\sA\on{-Gr}}(P_0\otimes_{\sA} V, W(2g))\right]
\ed
for its degree $0$ graded piece.  As in \cite{McNKirwan}, using Lemma \ref{adjunction identities} we may identify $\mathsf{Ext}$ with:
\begin{equation}\label{eq:perfect-complex}
L(V_{0},W_{0}) \xrightarrow{\partial_0} E(V_{0}, W_{1})\xrightarrow{\partial_1} L(V_{0}, W_{2g}),
\end{equation}
where $\partial_0 = \beta^\vee_0$ and $\partial_1 = \alpha^\vee_0$.  
\begin{prop}\label{prop:Ext properties}
Suppose that $\tau_{[0,N]}V$ and $W$ are graded $\sA\intN$-modules.  Then:
\begin{enumerate}
\item We have an isomorphism $\on{coker}(\partial_1) \cong \Hom_k\big(\Hom_{\sA\intN\on{-Gr}}(W,\tau_{[0,N]}V), k\big)$. 
\end{enumerate}
If, in addition,  $\tau_{[0,N]}V$  is $\theta$-stable and $W$ is $\theta$-semistable, both of dimension vector $\algtr$, then: 
\begin{enumerate}
\item[(2)] We have $\on{ker}(\partial_0) = 0$ unless $\tau_{[0,N]}V\cong W$, in which case $\on{ker}(\partial_0) \cong k$. 
\item[(3)]  We have that $\on{coker}(\partial_1)$ is zero unless $\tau_{[0,N]}V\cong W$, in which case $\on{coker}(\partial_1)\cong k$.
\end{enumerate}
\end{prop}
\begin{proof}
Assertion (2) follows from the exactness of \eqref{Hom sequence beginning} and stability.  Similarly, assertion (3) is immediate from assertion (1) by stability of $\tau_{[0,N]}V$ and semistability of $W$.  

Thus it remains to prove assertion (1).
Similarly to Lemma \ref{adjunction identities}, we use Proposition \ref{Hom equals tensor} to identify
\begin{align}\label{first adjunction identity}
\Hom_{\sA\on{-Gr}}(P_0\otimes_{\sA} V, W(2g)) & \cong V^*\otimes_{S_t} W\cong V_0^*\otimes_S W_{2g}\cong \Hom_S(V_0, W_{2g}),\\
\label{second adjunction identity}
\Hom_{\sA\on{-Gr}}(P_1\otimes_{\sA} V, W(1)) & \cong (B\otimes_S V_0)^*\otimes_S  W_1 \cong \Hom_S(B\otimes_S V_0, W_1).
\end{align}
Specifically, we use \eqref{first adjunction identity} to identify $\sum_r \lambda_r\otimes w_r\in V_0^* \otimes_S W_{2g}$ with an element $\phi\in L(V_0,W_{2g})$, i.e., an $I$-graded homomorphism $(\phi_i): V_0\rightarrow W_{2g}$; and we use \eqref{second adjunction identity}
to identify
$\sum_r\lambda_r\otimes w_r \in (B\otimes_S V_0)^*\otimes_S  W_1$ with an element 
$\psi \in E(V_0, W_1)$.  
Under these identifications, the elements 
\bd
\sum_r \lambda_r\big(a_j^* Dt^{-2}\eta_{s(a_j)}^\vee  - Dt^{-2}\eta_{t(a_j)}^\vee  a_j^*\big)w_r,  \hspace{2em}
\sum_r \lambda_r\big(a_j Dt^{-2}\eta_{t(a_j)}^\vee - Dt^{-2}\eta_{s(a_j)}^\vee  a_j\big)w_r
\ed
 of Proposition \ref{big formulas} are identified with
\bd
 \psi_{a_j} a_j^* t^{-2}D - a_j^*\psi_{a_j}t^{-2}D
 \hspace{2em} \text{and} \hspace{2em} 
\psi_{a_j^*} a_j t^{-2}D - a_j\psi_{a_j^*}t^{-2}D
\ed
for $\psi\in E(V_0,W_1) \cong \Hom_{\sA\on{-Gr}}(P_1\otimes_{\sA} V,W(1))$.

Via the trace pairings, the $k$-linear dual of $\partial_1$ is a map
$L(W_{2g}, V_0) \xrightarrow{\partial_1^*} E(W_1,V_0)$;
an element $\phi^*\in L(W_{2g}, V_0)$ satisfies $\partial_1^*(\phi^*) = 0$ only if
\bd
\on{tr}\left[\phi^* \psi_{a_j} a_j^* t^{-2}D - \phi^* a_j^*\psi_{a_j}t^{-2}D\right] = 0 
\hspace{1em} \text{and} \hspace{1em}
\on{tr}\left[\phi^* \psi_{a_j^*} a_j t^{-2}D - \phi^* a_j\psi_{a_j^*}t^{-2}D\right] = 0
\ed
for all $\psi\in E(V_0,W_1)$.  Since each $G_{a_j}t^{-2}$ acts as an isomorphism on $V^*$,  the elements $\lambda G_{a_j}t^{-2}\eta_{a_j}w$ and  $\lambda G_{a_j^*}t^{-2}\eta_{a_j^*}w$, for $\lambda\in V_0^*, w\in W_1$, collectively generate $\Hom_{\sA\on{-Gr}}(P_1\otimes_{\sA} V, W(1))$; it follows that  
an element $\phi^*\in L(W_{2g}, V_0)$ satisfies $\partial_1^*(\phi^*) = 0$ if and only if the above conditions are satisfied for all $\psi\in E(V_0,W_1)$.

Cyclically permuting, these conditions become
\begin{equation}\label{phi star conditions}
a_j^*t^{-2}D\phi^* - t^{-2}D\phi^*a_j^* = 0 
\hspace{1em} \text{and} \hspace{1em}
a_jt^{-2}D\phi^* - t^{-2}D\phi^*a_j = 0.
\end{equation}
Given $\phi^*\in L(W_{2g}, V_0)$ satisfying these conditions, define $\Phi^*: W\rightarrow \tau_{[0,N]} V$ by taking
$\Phi^*|_{W_{2g-m}} = t^{-m}D\phi^*t^{m}$.  It is immediate from the conditions \eqref{phi star conditions} that on $W_{2g-m}$, $m\geq 2$, we have that $\Phi^*$ commutes with all $a_j$ and $a_j^*$, whereas for $m=1$ we may write $\Phi^*|_{W_{2g-1}} = t t^{-2}D\phi^*t$ and again $\Phi^*$ commutes with $a_j, a_j^*$.  Thus $\Phi^*$ defines an $\sA\intN$-linear homomorphism $W\rightarrow \tau_{[0,N]} V$, yielding a linear map $\on{ker}(\partial_1^*)\hookrightarrow \Hom_{\sA\intN\on{-Gr}}(W, \tau_{[0,N]} V)$.  Conversely, given a graded $\sA\intN$-module homomorphism $\Phi^*: W\rightarrow \tau_{[0,N]} V$, defining $\phi^*: W_{2g}\rightarrow V_0$ by $\phi^* = D\inv\Phi^*|_{W_{2g}}$, we see that $\phi^*\in\on{ker}(\partial_1^*)$.  This completes the proof.
\end{proof}

\section{Cohomology of Varieties and Stacks}\label{sec:coh}
\begin{center}
{\bf
In the remainder of the paper, the base field $k$ is assumed to be $\mathbb{C}$. 
}
\end{center} 
Here as throughout the paper, we use $H^*(X)$ to denote cohomology with $\mathbb{Q}$-coefficients, and $H^{\on{BM}}_*(X)$ to denote Borel-Moore homology with $\mathbb{Q}$-coefficients; if $X$ is a smooth Deligne-Mumford stack, there is a canonical isomorphism $H^*(X) \cong H^{\on{BM}}_*(X)$.
\subsection{Mixed Hodge Structure on the Cohomology of an Algebraic Stack}
Suppose that $\mathfrak{X}$ is an algebraic stack of finite type over $\mathbb{C}$.  It follows from Example 8.3.7 of
\cite{hodge-III} that the cohomology $H^*(\mathfrak{X})$ comes equipped with a functorial mixed Hodge structure.
\begin{prop}
Suppose $\mathfrak{X}$ is a complex Deligne-Mumford stack with the action of the commutative group stack $BH$ for some finite group $H$, and that $\mathfrak{X}$ has a coarse moduli space $\mathfrak{X}\rightarrow\on{sp}(\mathfrak{X})$ with an isomorphism $\mathfrak{X}\rightarrow\on{sp}(\mathfrak{X}) = \mathfrak{X}/BH$.  Then 
$H^*(\mathfrak{X},\mathbb{Q}) = H^*\big(\on{sp}(\mathfrak{X}),\mathbb{Q}\big)$ as mixed Hodge structures.\footnote{We explicitly write the $\mathbb{Q}$-coefficients to emphasize that they are essential.}
\end{prop}
\begin{proof}
Use the Leray spectral sequence and the fact that $H^*(BH, \mathbb{Q}) = \mathbb{Q}$ for a finite group $H$.
\end{proof}

\subsection{Pushforwards and the Projection Formula}  
Suppose $f: X\rightarrow Y$ is a proper morphism of relative dimension $d$ of smooth, connected Deligne-Mumford stacks.  Then there is a pushforward, or Gysin, map $f_*: H^*(X)\rightarrow H^{*-d}(Y)$.  
\begin{prop}[\cite{dCM}]\label{dCM}
If $X$ and $Y$ are of finite type (so their cohomologies support canonical mixed Hodge structures), the Gysin map $f_*$ is a morphism of mixed Hodge structures.
\end{prop}
The Gysin map satisfies the projection formula: for classes $c\in H^*(X), c'\in H^*(Y)$, we have
\begin{equation}\label{projection formula}
f_*(c\cup f^*c') = f_*(c)\cup c'.
\end{equation}
Suppose $X$ and $Y$ are smooth Deligne-Mumford stacks and $C\in H^*(X\times Y)$ is a cohomology class.  By the K\"unneth theorem we have $H^*(X\times Y)\cong H^*(X)\otimes H^*(Y)$, and thus we may write $C = \sum x_i\otimes y_i$ with $x_i\in H^*(X)$, $y_i\in H^*(Y)$.  The classes $x_i$, $y_i$ are the {\em K\"unneth components} of $C$ (with respect to $X$ or $Y$ respectively).  

Now suppose that $f: X\rightarrow Y$ is a representable morphism from a smooth Deligne-Mumford stack $X$ to a smooth, proper Deligne-Mumford stack $Y$.  
The graph morphism $X\xrightarrow{(1,f)} X\times Y$ is not usually a closed immersion.  

\begin{prop}[cf. Proposition 2.1 of \cite{McNKirwan}]\label{prop:image}
The image of $f^*: H^*(Y)\rightarrow H^*(X)$ is contained in the span of the K\"unneth components of $(1,f)_*[X]$ with respect to the left-hand factor $X$.
\end{prop}
\begin{proof}
Write $\xymatrix{X &  X\times Y \ar[l]_{p_X} \ar[r]^{p_Y} & Y}$ for the projections.    Write $p_*: Y\rightarrow \on{Spec}(\mathbb{C})$ for the projection to a point; then $(p_X)_*$ exists since $Y$ is proper.  We have $f^* = (1,f)^*p_Y^*$ and $(p_X)_* (1,f)_* = \on{id}$.  Using the projection formula, then, we get
\bd
f^* = (p_X)_*(1,f)_*f^* = (p_X)_*(1,f)_*(1,f)^*p_Y^* = (p_X)_*\big((1,f)_*[X] \cap p_Y^*(-)\big).
\ed
This proves the claim.
\end{proof}

\subsection{Cohomology of Compactifications}
A finite-type Deligne-Mumford stack $\resol$ is {\em quasi-projective} if its coarse space $\on{sp}(\resol)$ is a quasi-projective scheme.  For example, if a reductive group $\bS$ acts on a polarized quasiprojective variety $\mathbb{M}$, then any open substack of $\mathbb{M}^s/\bS$ is a quasi-projective Deligne-Mumford stack.\footnote{Here $\mathbb{M}^s$ means stable points in the GIT sense: in particular, stabilizers are finite.}

The cohomology $H^k(M)$ is {\em pure} if its mixed Hodge structure is pure of weight $k$: that is, $W_k\big(H^k(M)\big) = H^k(M)$.  We say {\em $H^*(M)$ is pure} if each $H^k(M)$ is pure. 
\begin{prop}\label{prop:semiproj-coh}
Suppose $\mathfrak{Y} = Y/\mathbb{G}$ is a quotient stack (i.e., the quotient of an algebraic space by a linear algebraic group scheme) and that $\resol^\circ\subset\resol\subset\mathfrak{Y}$ are open, separated, quasi-projective, smooth Deligne-Mumford substacks of $\mathfrak{Y}$.  Then the image of the restriction map $H^k(\resol)\rightarrow H^k(\resol^\circ)$ contains $W_k\big(H^k(\resol^\circ)\big)$; in particular, if $H^*(\resol^\circ)$ is pure, then the restriction map is surjective.
\end{prop}
\begin{proof}
Consider first the case of smooth quasi-projective varieties $\resol^\circ\subset\resol$.  Then, for any smooth projective compactification $\overline{\resol}$ of $\resol$, the image of $H^*(\overline{\resol})\rightarrow H^*(\resol^\circ)$ is independent of the choice of $\overline{\resol}$: for example, by the Weak Factorization theorem, any two such $\overline{\resol}, \overline{\resol}'$ are related by a sequence of blow-ups and blow-downs along smooth centers in the complement of $\resol^\circ$, and the claimed independence follows from the usual formula for the cohomology of a blow-up.  Since the image of $H^k(\overline{\resol})$ in $H^k(\resol^\circ)$ is $W_k\big(H^*K(\resol^\circ)\big)$ by Corollaire 3.2.17 of \cite{Deligne}, the claim follows in this case.  

We now consider the general case.  By the assumptions, $\resol$ and $\resol^\circ$ are (separated) quasi-projective smooth Deligne-Mumford stacks that are global quotients.  By Theorem 1 of \cite{KreschVistoli},  there exist a smooth quasi-projective scheme $\mathsf{W}$ and a finite flat LCI morphism $\mathsf{W}\rightarrow \resol$; the fiber product $\resol^{\circ}\times_{\resol} \mathsf{W} \rightarrow \resol^{\circ}$ is then also finite, flat, and LCI.  Using the commutative square
\bd
\xymatrix{\resol^{\circ}\times_{\resol} \mathsf{W}  \ar[r]^{\hspace{2em}\wt{j}} \ar[d]^{q^\circ} & \mathsf{W}\ar[d]^{q}\\
\resol^{\circ} \ar[r]^{j} & \resol
}
\ed
and base change, we find:
\begin{enumerate}
\item $H^k(\mathsf{W}) \xrightarrow{q_*} H^k(\resol)$ and $H^k(\resol^{\circ}\times_{\resol} \mathsf{W})\xrightarrow{q^\circ_*} H^k(\resol^\circ)$ are surjective (indeed, $q_*q^*$ and $q^\circ_*(q^\circ)^*$ are multiplication by the degree of $q$).
\item Since the Gysin maps $q^\circ_*, q_*$ are morphisms of mixed Hodge structures by Proposition \ref{dCM},
\bd
W_k\big(H^k(\resol^{\circ}\times_{\resol} \mathsf{W})\big)\xrightarrow{q^\circ_*} W_k\big(H^k(\resol^\circ)\big) \hspace{1em} \text{is surjective}.
\ed
\item The image of $H^k(\mathsf{W})$ in $H^k(\resol^{\circ}\times_{\resol} \mathsf{W})$ contains $W_k\big(H^k(\resol^{\circ}\times_{\resol} \mathsf{W})\big)$, by the conclusion of the previous paragraph.
\end{enumerate}
The assertion is now immediate.
\end{proof}

\subsection{Markman's Formula for Chern Classes of Complexes}
Suppose that $\mathfrak{M}$ is a smooth Deligne-Mumford stack and 
\begin{equation}\label{Markman complex}
C: V_{-1}\xrightarrow{g} V_0\xrightarrow{f} V_1
\end{equation} is a complex of locally free sheaves on $\mathfrak{M}$ of ranks $r_{-1}, r_0, r_1$ respectively.
\begin{prop}[Lemma 4 of \cite{Markman}]\label{prop:Markman}
Suppose that $\Gamma\subset \mathfrak{M}$ is a smooth closed substack of pure codimension $m$, and that
the complex $C$ of \eqref{Markman complex} satisfies:
\begin{enumerate}
\item $\mathcal{H}^{-1}(C) = 0$, 
\item $\mathcal{H}^1(C)$ and $\mathcal{H}^1(C^\vee)$ are line bundles on $\Gamma$,
\item $m\geq 2$ and $\on{rk}(C) = m-2$.
\end{enumerate}
Then if $m$ is even, $c_m(C) = [\Gamma]$ and $c_m\big(\mathcal{H}^0(C)\big) = \left(1-(m-1)!\right)[\Gamma]$.
\end{prop}
\begin{remark}
Markman's Lemma 4 is ostensibly stated for smooth varieties $M$, but Section 3 of {\em op. cit.} generalizes the assertion to smooth Deligne-Mumford stacks.
\end{remark}

\subsection{Proofs of Theorems \ref{stack main thm} and \ref{main thm}}
Fix a quiver $Q$, stability condition $\theta$ for $\Qdbl$ and the corresponding stability condition $\thetagtr$ for $\Qgtr$ as in 
Section \ref{sec:semistability}.  Choosing a subgroup $\bS\subset\mathbb{G}$ as in Section \ref{sec:moduli stacks}, we obtain a  ``graph immersion'' in a product of Deligne-Mumford stacks 
\begin{equation}
\on{Rep}(\Lambda^q,\alpha)^{\theta\on{-s}}/\bS
\xrightarrow{\iota}
\on{Rep}(\Lambda^q,\alpha)^{\theta\on{-s}}/\bS \times \on{Rep}_{\on{gr}}(\sA\intN, \algtr)^{\thetagtr\on{-ss}}/\Sgtr.
\end{equation}
We write $\iota$ for the immersion and $\Gamma = \im(\iota)$ for its image, a smooth closed substack.  We remark that $\iota$ is {\em not} a closed immersion unless $H$ is trivial; however, the morphism $\iota$ identifies
\bd
\Gamma \cong \on{Rep}(\Lambda^q,\alpha)^{\theta\on{-s}}/\bS\times BH.
\ed 
It follows that $(1\times \iota)_*[\on{Rep}(\Lambda^q,\alpha)^{\theta\on{-s}}]$ is a nonzero rational multiple of $[\Gamma]$, and thus we may apply Proposition \ref{prop:image} with $(1\times \iota)_*[\on{Rep}(\Lambda^q,\alpha)^{\theta\on{-s}}]$ replaced by 
$[\Gamma]$, and we do this below.

The factors 
$\on{Rep}(\Lambda^q,\alpha)^{\theta\on{-s}}/\bS$ and  $\on{Rep}_{\on{gr}}(\sA\intN, \algtr)^{\thetagtr\on{-ss}}/\Sgtr$
come equipped with universal representations $V$, $W$ respectively.  The complex $\mathsf{Ext}$ defined in  
Section \ref{sec:Ext complex} descends to the product 
$\on{Rep}(\Lambda^q,\alpha)^{\theta\on{-s}}/\bS \times \on{Rep}_{\on{gr}}(\sA\intN, \algtr)^{\thetagtr\on{-ss}}/\Sgtr$.
We recall from Proposition \ref{prop:compactification} the compactification $\overline{\on{Rep}(\Lambda^q,\alpha)^{\theta\on{-s}}/\bS}$ of $\on{Rep}(\Lambda^q,\alpha)^{\theta\on{-s}}/\bS$, which maps to 
$\on{Rep}_{\on{gr}}(\sA\intN, \algtr)^{\thetagtr\on{-ss}}/\Sgtr$ and induces an isomorphism on the open substack
$\on{Rep}(\Lambda^q,\alpha)^{\theta\on{-s}}/\bS$.  Pulling the complex $\mathsf{Ext}$ back to the product
$\on{Rep}(\Lambda^q,\alpha)^{\theta\on{-s}}/\bS
\times
\overline{\on{Rep}(\Lambda^q,\alpha)^{\theta\on{-s}}/\bS}$,
we get a complex
 that we will denote $C$.  
 
Direct calculation shows that the rank of $C$ is $m-2 = \on{codim}(\Gamma) -2$ (we note that its rank depends only on $Q$ and $\alpha$: only the differentials distinguish between the ordinary and multiplicative preprojective algebras).  It follows from Proposition \ref{prop:Ext properties} that $C$ has the following properties:
\begin{enumerate}
\item $\mathcal{H}^{-1}(C) = 0$,
\item $\mathcal{H}^1(C)$ and $\mathcal{H}^1(C^\vee)$ are set-theoretically supported on $\Gamma$, and their scheme-theoretic restrictions to $\Gamma$ are line bundles.  
\end{enumerate}
Thus, in order to show that $\Gamma$ satisfies the hypotheses of Proposition \ref{prop:Markman}, it suffices to show that $\Gamma$ is the scheme-theoretic support of both $\mathcal{H}^1(C)$ and $\mathcal{H}^1(C^\vee)$.  We do this by considering a morphism 
\bd
\on{Spec}(k[\epsilon])\rightarrow 
\on{Rep}(\Lambda^q,\alpha)^{\theta\on{-s}}/\bS
\times
\overline{\on{Rep}(\Lambda^q,\alpha)^{\theta\on{-s}}/\bS}
\ed
(where here and throughout the remainder of the proof, $k[\epsilon]$ denotes the ring of dual numbers) with the property that the closed point maps to $\Gamma$.  Then it will suffice to show that either $\on{Spec}(k[\epsilon])$ maps scheme-theoretically to $\Gamma$, or that the pullbacks of 
$\mathcal{H}^1(C)$ and $\mathcal{H}^1(C^\vee)$ to $\on{Spec}(k[\epsilon])$ are scheme-theoretically supported at 
$\on{Spec}(k) \subset \on{Spec}(k[\epsilon])$.

We thus consider a representations $\overline{V}_\epsilon, \overline{V}'_{\epsilon}$ of $\Lambda^q[\epsilon]$ that are flat over $k[\epsilon]$  and having dimension vector $\alpha$ after tensoring with $k\otimes_{k[\epsilon]}-$;  and let $V_\epsilon =\overline{V}_\epsilon[t]$, 
$V_\epsilon' =\overline{V}'_\epsilon[t]$.   Assume $\tau_{[0,N]}V_{\epsilon}$, $\tau_{[0,N]}V_{\epsilon}'$ are $\thetagtr$-stable.  The complex $C_\epsilon$ defined as in  \eqref{eq:perfect-complex} becomes a complex of free $k[\epsilon]$-modules, and $\mathcal{H}^{-1}(C_\epsilon) = \Hom_{\sA_\epsilon\on{-Gr}}\big(\tau_{[0,N]}V_{\epsilon}, \tau_{[0,N]}V_{\epsilon}'\big)$.  This cohomology is isomorphic to $k[\epsilon]$ if and only if $\tau_{[0,N]}V_{\epsilon}\cong \tau_{[0,N]}V_{\epsilon}'$.  Thus,
$\mathcal{H}^1(C_\epsilon^\vee)$ is isomorphic to $k[\epsilon]$ if and only if $\tau_{[0,N]}V_{\epsilon}\cong \tau_{[0,N]}V_{\epsilon}'$.  It follows that the scheme-theoretic support of $\mathcal{H}^1(C^\vee)$ is the reduced diagonal $\Gamma$.  

It remains to check that the same is true of $\mathcal{H}^1(C)$.  
To do that, we again start with $\tau_{[0,N]}V_{\epsilon}$, $\tau_{[0,N]}V_{\epsilon}'$ as above, but consider them as graded $\sA$-modules (i.e., forgetting the $k[\epsilon]$-module structure) and form the complex $C$.  
Assume without loss of generality that $k\otimes_{k[\epsilon]}\tau_{[0,N]}V_{\epsilon}\cong k\otimes_{k[\epsilon]}\tau_{[0,N]}V_{\epsilon}'$ as graded $\sA$-modules.  
We have a short exact sequence of graded $\sA$-modules
\begin{equation}\label{eq:extension}
0\rightarrow \epsilon\tau_{[0,N]}V_{\epsilon} \rightarrow \tau_{[0,N]}V_{\epsilon} \rightarrow k\otimes_{k[\epsilon]}\tau_{[0,N]}V_{\epsilon} \rightarrow 0,
\end{equation}
where by $k[\epsilon]$-flatness we have $ \epsilon\tau_{[0,N]}V_{\epsilon} \cong k\otimes_{k[\epsilon]}\tau_{[0,N]}V_{\epsilon}$, both stable; and similarly for $V'$.  
Assume without loss of generality that $k\otimes_{k[\epsilon]}\tau_{[0,N]}V_{\epsilon}\cong k\otimes_{k[\epsilon]}\tau_{[0,N]}V_{\epsilon}'$ as graded $\sA$-modules.  Suppose there is a nonzero map of graded $\sA$-modules,
$\phi: \tau_{[0,N]}V_{\epsilon}\rightarrow \tau_{[0,N]}V_{\epsilon}'$.  If the composite
\begin{equation}\label{eq:composite}
\epsilon\tau_{[0,N]}V_{\epsilon} \hookrightarrow \tau_{[0,N]}V_{\epsilon} \xrightarrow{\phi} \tau_{[0,N]}V_{\epsilon}'
\twoheadrightarrow
k\otimes_{k[\epsilon]}\tau_{[0,N]}V_{\epsilon}'
\end{equation}
 is nonzero, it is an isomorphism, since both its domain and target are stable of dimension vector $\algtr$; in which case both
 \eqref{eq:extension} and its analogue for $\tau_{[0,N]}V_{\epsilon}'$ are split extensions.  This means that the tangent vector to 
$\on{Rep}(\Lambda^q,\alpha)^{\theta\on{-s}}/\bS \times \on{Rep}(\Lambda^q,\alpha)^{\theta\on{-s}}/\bS$ determined by $(\overline{V}_\epsilon, \overline{V}_\epsilon')$ is zero, and thus irrelevant to our analysis of the scheme-theoretic support of $\mathcal{H}^1(C)$.  Thus we may assume that  the composite \eqref{eq:composite} is zero, and so the morphism
$\phi$ is a homomorphism of $1$-extensions.  Now if $\phi(\epsilon\tau_{[0,N]}V_{\epsilon})\neq 0$, then again by stability it maps isomorphically onto $\epsilon\tau_{[0,N]}V_{\epsilon}'$.  Since \eqref{eq:extension} is non-split, it follows that $\phi$ is an isomorphism, implying that the tangent vector determined by $(\overline{V}_\epsilon, \overline{V}_\epsilon')$ is tangent to $\Gamma$, and again irrelevant to our analysis of the scheme-theoretic support of $\mathcal{H}^1(C)$.  Finally then, we may assume that $\phi(\epsilon\tau_{[0,N]}V_{\epsilon}) = 0$.  It follows that $\phi$ factors through the quotient $k\otimes_{k[\epsilon]}\tau_{[0,N]}V_{\epsilon}$; similarly its image lies in $\epsilon\tau_{[0,N]}V_{\epsilon}'$.  It follows that  $\Hom_{\sA}\on{-Gr}\big(\tau_{[0,N]}V_{\epsilon}, \tau_{[0,N]}V_{\epsilon}'\big)$ is scheme-theoretically supported over 
$\spec(k)\subset \spec k[\epsilon]$, and hence by Proposition \ref{prop:Ext properties}(1) that the same is true of $\mathcal{H}^1(C)$.  Since this is true for every $\spec k[\epsilon]\rightarrow \on{Rep}(\Lambda^q,\alpha)^{\theta\on{-s}}/\bS \times \on{Rep}(\Lambda^q,\alpha)^{\theta\on{-s}}/\bS $ not tangent to $\Gamma$, we conclude that $\mathcal{H}^1(C)$ has scheme-theoretic support equal to $\Gamma$, as required.

By Proposition \ref{prop:Markman}, then, we conclude that $[\Gamma] = c_m(C)$.  By Proposition \ref{prop:image}, the K\"unneth components of $c_m(C)$ thus span the image of the restriction map 
\bd
H^*\big(\overline{\on{Rep}(\Lambda^q,\alpha)^{\theta\on{-s}}/\bS}\big)
\longrightarrow
H^*\big(\on{Rep}(\Lambda^q,\alpha)^{\theta\on{-s}}/\bS\big),
\ed
which by  Proposition \ref{prop:semiproj-coh} is exactly 
$\oplus_m W_m\Big(H^m\big(\on{Rep}(\Lambda^q,\alpha)^{\theta\on{-s}}/\bS\big)\Big)$.  Since the Chern classes of $C$ are polynomials in the Chern classes of the tautological bundles (see the proof of Proposition 2.4(ii) of \cite{McNKirwan}), 
this completes the proof of Theorem \ref{stack main thm}, hence also of Theorem \ref{main thm}.\hfill\qedsymbol

\subsection{Proof of Theorem \ref{derived cat}}
The proof of Theorem \ref{derived cat} is essentially identical to that of Theorem 1.6 of \cite{McNKirwan} (and we note that Theorem \ref{derived cat} holds whenever $k$ is any field of characteristic zero and $q\in k^\times$).    Indeed, the assumption that there is a vertex $i_0\in I$ for which $\alpha_{i_0}=1$ guarantees the following.  First, we may take $\mathbb{S} = \prod_{i\neq i_0} GL(\alpha_i)$, which acts freely on the stable locus: thus, $\mathcal{M}_\theta^q(\alpha)^{\on{s}}$ is a fine moduli space for stable representations of $\Lambda^q$.
Second, exactly as in the proof of Theorem 1.6 of \cite{McNKirwan}, in the complex \eqref{eq:perfect-complex}, there are direct sum decompositions 
\bd
L(V_0,W_0) = \Hom(V_{0, i_0}, W_{0, i_0}) \oplus \big(\oplus_{i\neq i_0} \Hom(V_{0,i}, W_{0,i})\big) \; \text{and} \; 
\ed
\bd
L(V_0, W_{2g}) = \Hom(V_{0,i_0}, W_{2g, i_0})   \oplus \big(\oplus_{i\neq i_0} \Hom(V_{0,i}, W_{2g,i})\big),
\ed
so that the complex obtained by modifying \eqref{eq:perfect-complex} given by
\bd
\oplus_{i\neq i_0} \Hom(V_{0,i}, W_{0,i}) \xrightarrow{\partial_0} E(V_0,W_1) \xrightarrow{\partial_1} L(V_0,W_{2g})/\Hom(V_{0,i_0}, W_{2g, i_0})
\ed
has no cohomology at the ends, and in the middle has cohomology $\mathcal{H}$ that is a rank $m = \on{codim}(\Gamma)$ vector bundle.  Moreover, the remaining map $k=\Hom(V_{0, i_0}, W_{0, i_0}) \rightarrow E(V_0,W_1)$ defines a section $s$ of $\mathcal{H}$ whose scheme-theoretic zero locus is $Z(s) = \Gamma$.  The remainder of the proof now copies that of
Theorem 1.6 of \cite{McNKirwan}.\hfill\qedsymbol

\end{document}